\documentclass{amsart}
\usepackage{amsmath,amsfonts,amssymb,amsthm,mathtools}
\usepackage{amscd}
\usepackage{bbm}
\usepackage{enumerate}
\usepackage{galois}
\usepackage{mathrsfs}
\usepackage{xypic}
\usepackage{geometry}
\usepackage{hyperref}

\usepackage{color}

\theoremstyle{plain}
\newtheorem{Th}{Theorem}[section]
\newtheorem{Cor}[Th]{Corollary}
\newtheorem{Lem}[Th]{Lemma}
\newtheorem{Prop}[Th]{Proposition}

\newtheorem{Def}{Definition}[section]

\newtheorem{Rem}{Remark}[section]

\numberwithin{equation}{section}

\DeclareMathOperator{\ima}{Im}
\DeclareMathOperator{\cok}{coker}

\newcommand{\der}{\mathrm{Der}(\hm A)}
\newcommand{\derx}{\mathrm{Der}^{\qp_x}(\hm A)}
\newcommand{\diff}[2]{\frac{\partial #1}{\partial #2}}

\newcommand{\qa}{\alpha}
\newcommand{\qb}{\beta}
\newcommand{\qd}{\delta}
\newcommand{\qg}{\gamma}
\newcommand{\qs}{\sigma}
\newcommand{\qt}{\tau}
\newcommand{\qth}{\theta}
\newcommand{\qe}{\varepsilon}

\newcommand{\qp}{\partial}
\newcommand{\qo}{\omega}

\newcommand{\Qg}{\Gamma}
\newcommand{\fm}{\mathfrak{m}}
\newcommand{\ql}{\lambda}

\newcommand{\Qd}{\Delta}

\newcommand{\res}{\mathrm{res}\,}
\newcommand{\vard}[2]{\frac{\delta #1}{\delta #2}}

\newcommand{\hm}[1]{\hat{\mathcal #1}}

\newcommand{\xym}[1]{\begin{center}\leavevmode \xymatrix{#1} \end{center}}

\newcommand{\kk}[1]{\left(#1\right)}

\newcommand{\fk}[2]{\left[#1, #2\right]}

\begin{document}

\title{Deformation of scalar generalized bi-Hamiltonian structure}

\author{Zhe Wang}
\address{Z.~Wang:\newline RIKEN Center for Interdisciplinary Theoretical and Mathematical Sciences (iTHEMS), RIKEN, Wako 351-0198, Japan}
\email{zhe.wang.aa@riken.jp}

\begin{abstract}
We study the classification of scalar generalized bi-Hamiltonian structures with semisimple leading terms. As an application, we prove the existence of integrable viscous conservation laws for any given viscous central invariant.  
\end{abstract}

\date{}
\maketitle
\tableofcontents


\section{Introduction}
The Witten-Kontsevich theorem which relates the intersection numbers over the moduli space of stable curves to the KdV hierarchy reveals the deep relationships between integrable hierarchies and 2D topological field theories (2dTFTs) \cite{kontsevich1992intersection,witten1990two}. Inspired by this, Dubrovin and Zhang developed a general theory toward the classification of integrable hierarchies associated with semisimple Frobenius manifolds \cite{dubrovin2001normal}. They showed that there exists a unique integrable hierarchy associated with a given semisimple Frobenius manifold, whose hydrodynamic leading term is the Principal Hierarchy of the Frobenius manifold, satisfying the four axioms: quasi-triviality, tau-symmetry, bihamiltonian structure and linearized Virasoro symmetries. In \cite{dubrovin2001normal}, they showed that the fourth axiom uniquely determines this integrable hierarchy, and the obtained hierarchy is called the Dubrovin-Zhang hierarchy associated with the Frobenius manifold. 

Initiated by the work \cite{lorenzoni2002deformations}, a series of works have been devoted to the study of Dubrovin-Zhang hierarchies in terms of the geometry of bi-Hamiltonian structures \cite{carlet2018deformations,DLZ-1,liu2013bihamiltonian,dubrovin2018bihamiltonian,getzler2002darboux,falqui2012exact, liu2023variational,liu2022variational, liu2025linearization}. In particular, the classification of bi-Hamiltonian structures with semisimple leading terms play a central role in the study of Dubrovin-Zhang hierarchies. It is shown that (the Miura equivalence class of) an $n$-comopnent bi-Hamiltonian structure with semisimple leading term is uniquely determined by $n$ functions of one variable, called the central invariants, and that the (Miura equivalence class of) Dubrovin-Zhang hierarchy associated with a semisimple Frobenius manifold is the unique integrable hierarchy obtained from the bi-Hamiltonian structure with the central invariants being $\frac{1}{24}$.

Parallel to the construction of Dubrovin-Zhang hierarchies from semisimple Frobenius manifolds, there are also different approaches to construct integrable hierarchies from 2dTFTs, namely the Givental-Teleman's reconstruction theory \cite{teleman2012structure,givental2001gromov,buryak2012polynomial} and the double ramification hierarchy \cite{buryak2015double}. These constructions are further generalized to the case for (bi-)flat F-manifolds \cite{arsie2021flat,arsie2023semisimple}. F-manifolds were originally introduced by Hertling and Manin \cite{hertling1999weak}, and (bi-)flat F-manifolds are F-manifolds with additional flat structures as introduced by Manin \cite{manin2005f}. These manifolds are closely related to the theory of integrable systems of hydrodynamic type \cite{lorenzoni2011f}, and both Givental-Teleman's reconstruction theory and the construction of double ramification hierarchies can be generalized to the case of flat F-manifolds \cite{arsie2021flat,arsie2023semisimple}.

However, the construction of Dubrovin-Zhang hierarchies cannot be generalized to the case of (bi-)flat F-manifolds straightforwardly, since the Principal Hierarchy of a (bi-)flat F-manifold does not possess a (bi)-Hamiltonian structure in general. In order to extend the construction of Dubrovin-Zhang hierarchies to the case of (bi-)flat F-manifolds and compare it with the other two constructions, in \cite{lorenzoni2026generalised}, Lorenzoni and the author introduced the notion of generalized (bi-)Hamiltonian structures of hydrodynamic type, which aims at extending the (bi-)Hamiltonian formalism to the Principal Hierarchies associated with (bi)-flat F-manifolds. The key point is to replace the usual notion of (bi-)Hamiltonian structures with mutually commuting odd flows on the infinite jet space of a super manifold. The generalized (bi-)Hamiltonian structures of hydrodynamic type are shown to be equivalent to (a compatible pair of) flat connections, and hence the Principal Hierarchy of a (bi-)flat F-manifold is naturally a generalized (bi-)Hamiltonian hierarchy of hydrodynamic type. 

In this paper, we study the classification of scalar generalized bi-Hamiltonian structures with semisimple leading terms, which may serve as the first step towards the construction of Dubrovin-Zhang hierarchies associated with semisimple bi-flat F-manifolds. Let us summarize the main results of this paper.

Denote by 
\[
\hm A = C^\infty(u)[[u^{(s+1)},\qth^s\colon s\geq 0]]
\] 
the ring of differential polynomials, where $\qth^s$ are odd variables satisfying
\[
\qth^s\qth^t+\qth^t\qth^s = 0,\quad s,t\geq 0.
\]
Define the derivation
\[
\qp_x = \sum_{s\geq 0}u^{(s+1)}\diff{}{u^{(s)}}+\qth^{s+1}\diff{}{\qth^s},\quad u^{(0)}:=u,
\]
then it follows that $\hm A$ can be viewed as the coordinate ring of the infinite jet bundle of a super manifold of dimension $(1|1)$ with local coordinates $(u,\qth)$, and hence sometimes we will use $u_{x}, u_{xx},\dots$ to denote $u^{(1)},u^{(2)},\dots$. The ring $\hm A$ admits two natural gradations, namely the differential degree given by  
\[
\deg_{\qp_x} u^{(s)} = s,\quad \deg_{\qp_x} \qth^s = s,\quad s\geq 0,
\]
and the super degree given by 
\[
\deg_\qth u^{(s)} = 0,\quad \deg_\qth \qth^s = 1,\quad s\geq 0.
\]
The subspace consisting of elements of differential degree $d$ is denoted by $\hm A_d$ and by $\hm A^p$ that for super degree $p$, we also introduce the notation $\hm A^p_d = \hm A^p\cap\hm A_d$. We consider the space
\[
\derx^p_d = \big\{X = \sum_{s\geq 0}\qp_x^s P\diff{}{u^{(s)}}+\qp_x^s Q\diff{}{\qth^s}\colon P\in\hm A^p_d,\ Q\in\hm A^{p+1}_d\big\},\quad p\geq -1,\quad d\geq 0
\]
whose elements are called evolutionary vector fields. The total space \[\derx = \bigoplus_{p,d}\derx^p_d\] is a graded Lie-algebra given by the usual graded commutators of vector fields. A scalar semisimple generalized bi-Hamiltonian structure is a pair of evolutionary vector fields $(P_0,P_1)$ in $\derx^1_1$ of the form  
\begin{align}
    \label{AB}
&P_0 = \sum_{s\geq 0}\qp_x^s\kk{\qth^1+ f(u)u_x\qth}\diff{}{u^{(s)}}+\qp_x^s\kk{f(u)\qth\qth^1}\diff{}{\qth^s},\\
\label{AC}
&P_1 = \sum_{s\geq 0}\qp_x^s\kk{u\qth^1+\kk{1+ug(u)}u_x\qth}\diff{}{u^{(s)}}+\qp_x^s\kk{ \kk{1+ug(u)}\qth\qth^1}\diff{}{\qth^s},
\end{align}
where $f(u)$ and $g(u)$ are arbitrary smooth functions. It is easy to check that $[P_i,P_j] = 0$ for $i,j = 0,1$. A deformation of $(P_0,P_1)$ is a pair $(\tilde P_0,\tilde P_1)$ given by 
\[
\tilde P_i = P_i+Q_i,\quad Q_i\in\derx^1_{\geq 2},\quad i=0,1,
\]
satisfying the condition 
\[
\fk{\tilde P_i}{\tilde P_j} = 0,\quad i,j = 0,1.
\]
Two deformations are said to be equivalent if there exists a change of variables of the form  
\[
u^{(s)}\mapsto \qp_x^s\kk{u+F},\quad \qth^s\mapsto \qp_x^s\kk{\qth+G},\quad F\in\hm A^0_{\geq 1},\quad G\in\hm A^1_{\geq 1}
\]
that transforms one deformation to the other. Such a change of variable is called a generalized Miura-type transformation.

Given the scalar semisimple generalized bi-Hamiltonian structure of the form \eqref{AB}, \eqref{AC}, let us define the function
\[
\bar K(u) = 1+u\left(g(u)-f(u)\right).
\]
We say \eqref{AB}, \eqref{AC} is of rigid type if $\bar K(u)$ is a non-constant function or if it is a constant irrational number. We say \eqref{AB}, \eqref{AC} is of $D$-type if $\bar K(u) = \frac12$, is of viscous type if $\bar K(u) = 1$, and all other cases are called of exceptional type.
\begin{Th}
    \label{AG}
    Fix a scalar semisimple generalized bi-Hamiltonian structure $(P_0,P_1)$ of the form \eqref{AB}, \eqref{AC}.
    \begin{enumerate}
        \item If $(P_0,P_1)$ is of rigid type, then all deformations are equivalent, hence are all equivalent to the trivial deformation $\tilde P_0 = P_0$ and $\tilde P_1 = P_1$.
        \item  If $(P_0,P_1)$ is of $D$-type, then equivalence classes of deformations are parametrized by a single function called the central invariant.
        \item If $(P_0,P_1)$ is of viscous type, then equivalence classes of deformations are parametrized by a single function called the viscous central invariant.
    \end{enumerate}
\end{Th}

We also have the uniqueness of symmetries for deformed generalized bi-Hamiltonian structure.
\begin{Th}
    \label{AH}
    Let $(\tilde P_0,\tilde P_1)$ be a deformation of the scalar semisimple generalized bi-Hamiltonian structure of the form \eqref{AB}, \eqref{AC}, and let $X\in\derx^0_{\geq 2}$ such that 
    \[
    \fk{\tilde P_0}{X} = \fk{\tilde P_1}{X} = 0.
    \]
    Then we have $X=0$. In particular, symmetries are uniquely determined by their hydrodynamic leading terms, and all symmetries mutually commute with each other.
\end{Th}

Let us present here some examples. A genuine scalar semisimple bi-Hamiltonian structure gives an example of generalized bi-Hamiltonian structures of $D$-type, and actually all generalized bi-Hamiltonian structures of $D$-type arise in this way. For example, we take $(\tilde P_0,\tilde P_1)$ whose actions are given by
\[
\tilde P_0(u) = \qth^1,\quad \tilde P_0(\qth) = 0,\quad \tilde P_0(u) = u\qth^1+\frac{1}{2}u_x\qth+c\qth^3,\quad \tilde P_1(\qth) = \frac12\qth\qth^1, 
\]
where $c$ is a constant, and different values of $c$ present non-equivalent deformations. Taking $c=\frac{1}{8}$ we recover the bi-Hamiltonian structure of the KdV hierarchy as in the Witten-Kontsevich theorem. An example of deformation of generalized bi-Hamiltonian structure of viscous type is 
\[
\tilde P_0(u) = \qth^1,\quad \tilde P_0(\qth) = 0,\quad \tilde P_0(u) = u\qth^1+u_x\qth+c\qth^2,\quad \tilde P_1(\qth) = \qth\qth^1, 
\] 
where $c$ is a constant, and different values of $c$ present non-equivalent deformations. Note that a major difference from the above D-type example is that the non-trivial deformations start from differential degree 2, which is a phenomenon that cannot be studied within the framework of bi-Hamiltonian structures. We will give a detailed study of this example in Sect.\,\ref{AD}. For generalized bi-Hamiltonian structure of exceptional type, the deformations can be very complicated, and we give an example in Sect.\,\ref{AE}. An important application of the above results is the proof of the existence of integrable viscous conservation laws. Recall that an integrable conservation law is an evolutionary PDE of the form 
\[
\diff{u}{t} = \qp_x\big(u^2+ \qe 2a(u)u_x+\sum_{n\geq 2}\qe^n B_n\big),\quad B_n\in\hm A^0_n, 
\]
which is formally integrable, meaning that for any smooth function $f(u)$ there exists a symmetry of the above PDE of the form 
\[
\diff{u}{t_f} = \qp_x\big(f(u)+ \sum_{n\geq 1}\qe^n C_n[f]\big),\quad C_n[f]\in\hm A^0_n. 
\]
An integrable conservation law is called viscous if $a(u)\neq 0$, which is studied in detail by Arsie, Lorenzoni and Moro \cite{arsie2015integrable}. In particular, they conjectured that Miura equivalence classes of integrable viscous conservation laws are parametrized precisely by the non-zero function $a(u)$. Using the generalized bi-Hamiltonian structure, we partially verify this conjecture by proving the following existence result.
\begin{Th}
    \label{AK}
    There exists an integrable conservation law for any non-zero function $a(u)$.
\end{Th}  

The paper is organized as follows. In Sect.\,\ref{AF}, we present necessary preliminaries regarding the formulation and the geometry of generalized (bi)-Hamiltonian structure of hydrodynamic type. Sect.\,\ref{AI} is devoted the proof of Theorem\,\ref{AG} and Theorem\,\ref{AH}, where we classify the deformations by computing  cohomology groups of a certain complex. We present some examples and applications in Sect.\,\ref{AJ} and prove Theorem\,\ref{AK}. Finally, we give some concluding remarks in Sect.\,\ref{AL}.

 \vspace{2em}
 \noindent\textbf{Acknowledgements} The author is supported by JSPS Grants-in-Aid for Scientific Research JP26K16995.

\section{Preliminaries}
\label{AF}
In this section, we introduce basic notions of infinite jet manifolds, and realize various constructions in integrable hierarchies as certain geometric structures on them. One may refer to \cite{DLZ-1,liu2018lecture} for details.
\subsection{Infinite jet manifolds and (bi-)Hamiltonian structures}
Let us fix a $\mathbb Z$-graded manifold $\hat M$ of dimension $(n|n)$.  Locally $\hat M$ is isomorphic, as a $\mathbb Z$-graded space, to $\mathbb R^n\oplus \mathbb R^n$, where the first component has degree 0 and the second degree 1. Hence, a local chart is given by an open subset $U\subseteq \hat M$ together with coordinates
\begin{equation}
    \label{AN}
(u^1,\dots, u^n;\qth_1,\dots,\qth_n),
\end{equation}
where $u^1,\dots, u^n$ are the usual even coordinates, and $\qth_1,\dots,\qth_n$ are odd coordinates satisfying 
\[
\qth_\qa\qth_\qb+\qth_\qb\qth_\qa = 0,\quad \qa,\qb = 1,\dots,n.
\]
Moreover, we require that all transition functions preserve the $\mathbb Z$-gradation, hence for another local chart on $\bar U\subseteq \hat M$ with coordinates 
\[
(\bar u^1,\dots, \bar u^n;\bar \qth_1,\dots,\bar \qth_n),
\]
we must have 
\begin{equation}
    \label{AM}
\bar u^\qa = \bar u^\qa(u^1,\dots, u^n),\quad \bar\qth_\qa = f_\qa^\qb(u^1,\dots,u^n)\qth_\qb,
\end{equation}
here and henceforth we assume the summation over repeated upper and lower Greek indices. Note that the even coordinates $u^\qa$ together with their transition functions give the atlas of an ordinary manifold of dimension $n$ which we denote by $\hat M_0$.

Denote by $J^\infty(\hat M)$ the infinite jet bundle of $\hat M$, which is a $\mathbb Z$-graded fiber bundle of dimension $(\infty|\infty)$ trivialized over local charts of $\hat M$. Take local coordinates of the form \eqref{AN} then we denote by $u^{\qa,s}$ for $s\geq 1$ the even fiber coordinates of $J^\infty(M)$ and by $\qth_\qa^s$ the odd fiber coordinates. The transition functions of fiber coordinates with respect to the coordinate changes of the form \eqref{AM} are recursively defined by 
\begin{equation}
    \label{AO}
\bar u^{\qa,s} = \sum_{i=0}^{s-1}\diff{\bar u^{\qa,s-1}}{u^{\qb,i}}u^{\qb,i+1},\quad \bar\qth_\qa^s = \sum_{i=0}^s\binom{s}{i}\diff{u^{\qb,i}}{\bar u^\qa}\qth_\qb^{s-i},\quad s\geq 1,
\end{equation} 
where we use the convention $u^{\qa,0} = u^\qa$ and $\qth_\qa^0 = \qth_\qa$. The transitions above imply that there exists a vector field with the local expression 
\[
\qp_x = \sum_{s\geq 0} u^{\qa,s+1}\diff{}{u^{\qa,s}}+\qth_\qa^{s+1}\diff{}{\qth_\qa^s},
\]
that is globally well-defined on the total space of $J^\infty(\hat M)$. It is easy to see that 
$u^{\qa,s} = \qp_x^s u^\qa$, and therefore sometimes we use $u^\qa_x,u^{\qa}_{xx},\dots$ to denote $u^{\qa,1},u^{\qa,2},\dots$.

Let us denote by $\hm A$ the ring of differential polynomials on $J^\infty(\hat M)$, and locally it is given by 
\[
\hm A = C^\infty(u^\qa)[[u^{\qa,s+1},\qth_\qa^s\colon s\geq 0]],
\]
furthermore, the action of $\qp_x$ endows $\hm A$ with the structure of a differential algebra.
The ring $\hm A$ comes with two natural $\mathbb Z$-gradations, one is called the differential degree given by assigning 
\[
\deg_{\qp_x} u^{\qa,s} = s,\quad \deg_{\qp_x}\qth_\qa^s = s,\quad s\geq 0,
\]
and the other is called the super gradation given by 
\[
\deg_{\qth} u^{\qa,s} = 0,\quad \deg_{\qth}\qth_\qa^s = 1,\quad s\geq 0.
\]
Due to the homogeneity of transitions \eqref{AM} and \eqref{AO}, it follows that these gradations are globally well-defined. The subspace consisting of homogeneous elements of differential degree $d$ is denoted by $\hm A_d$ and by $\hm A^p$ that for super degree $p$, and we also denote $\hm A^p_d = \hm A^p\cap\hm A_d$. 

We continue to define super derivations of $\hm A$. For a fixed $p\in\mathbb Z$, the space $\der^p$ consists of linear maps $X\colon \hm A\to \hm A$ satisfying 
\[
X(\hm A^q)\subseteq \hm A^{p+q},\quad X(fg) = X(f)g+(-1)^{pq}fX(g),\quad f\in \hm A^q,\quad g\in\hm A.
\]
Note that $\der^p = 0$ for $p\leq -2$, and we denote by $\der = \bigoplus_{p\geq -1}\der^p$.
This is a graded Lie algebra with respect to usual graded commutators. We also use the notations $\der_d$ to denote the subspace consisting derivations $X$ with $X(\hm A_n)\subseteq \hm A_{n+d}$ and denote $\der^p_d = \der^p\cap\der_d$.
We have $\qp_x\in\der^0_1$ and define the subspaces
\[
\derx^p = \big\{X\in\der^p,\quad [X,\qp_x] = 0\big\},\quad \derx = \bigoplus_{p\geq -1}\derx^p,
\]
then $\derx$ is a Lie subalgebra whose elements are called flows or evolutionary vector fields. 

For a smooth manifold $M$ of dimension $n$, we can construct canonically an infinite jet manifold $J^\infty(\hat M)$ by taking $\hat M = T^*M[1]$ with $\hat M_0 = M$. More explicitly, the transitions of $\hat M$ takes the form 
\begin{equation}
    \label{AP}
\bar u^\qa = \bar u^\qa(u^1,\dots, u^n),\quad \bar\qth_\qa =\diff{u^\qb}{\bar u^\qa}\qth_\qb,
\end{equation}
where $(u^\qa)$ and $(\bar u^\qa)$ are local coordinates of $M$. Define the quotient space  $\hm F = \hm A/\qp_x\hm A$ whose elements are called local functionals, and for an element $f\in \hm A$, we denote by $\int f$ its class in $\hm F$. The space $\hm F$ admits two gradations induced from those of $\hm A$, and we use notations $\hm F^p, \hm F_d$ and $\hm F^p_d$ for subspaces of homogeneous elements. $\hm F$ admits a graded Lie algebra structure induced from the natural symplectic structure on $T^*M$. Locally the Lie bracket takes the form 
\begin{equation}
    \label{AU}
[F,G]=\int \vard{F}{\qth_\qa}\vard{G}{u^\qa}+(-1)^p \vard{F}{u^
\qa}\vard{G}{\qth_\qa}, \quad F\in\hm F^p,\quad G\in\hm F, 
\end{equation}
where the variational derivatives are defined by 
\[
\vard{F}{u^\qa} = \sum_{s\geq 0}(-\qp_x)^s\diff{f}{u^{\qa,s}},\quad \vard{F}{\qth_\qa}= \sum_{s\geq 0}(-\qp_x)^s\diff{f}{\qth_\qa^s},\quad F = \int f.
\]
It can be checked that variational derivatives as well as the above Lie bracket are globally well-defined with respect to the transitions \eqref{AP}. An important construction is that we can embed $\hm F$ into $\derx$ as a Lie subalgebra by 
\begin{equation}
    \label{AQ}
\hm F^p\to \derx^{p-1}\colon P\mapsto D_P = \sum_{s\geq 0}\qp_x^s\kk{\vard{P}{\qth^\qa}}\diff{}{u^{\qa,s}}+(-1)^p \kk\qp_x^s{\vard{P}{u^\qa}}\diff{}{\qth_\qa^s}.
\end{equation}
Moreover, we have the identity 
\begin{equation}
    \label{AW}
    \int D_P(Q) = [P,Q],\quad \forall\ P,Q\in\hm F.
\end{equation}
Now we define a Hamiltonian structure to be a local functional $P\in\hm F^2$ such that $[P,P] = 0$ and a bi-Hamiltonian structure to be a pair $(P_0,P_1)$ of Hamiltonian structures satisfying $[P_0,P_1] = 0$. A Hamiltonian structure $P$ hence defines an odd flow $D_P\in\derx^1$ commuting with itself, and similarly a bi-Hamiltonian structure $(P_0,P_1)$ determines a pair of flows $(D_{P_0},D_{P_1})$ satisfying $[D_{P_i},D_{p_j}] = 0$.

\subsection{Generalized (bi-)Hamiltonian structures of hydrodynamic type}
Let $\hat M$ be a $\mathbb Z$-graded manifold of dimension $(n|n)$. We recall in this section the basic theory of generalized (bi)-Hamiltonian structures of hydrodynamic type as introduced and studied in \cite{lorenzoni2026generalised}.
\begin{Def}
    A generalized Hamiltonian structure on $\hat M$ is an odd flow $X\in\derx^1$ such that $[X,X] = 0$. A generalized bi-Hamiltonian structure is a pair $(X_0,X_1)$ of generalized Hamiltonian structures satisfying $[X_0,X_1] = 0$.
\end{Def}

It follows that these notations generalize the usual notion of (bi-)Hamiltonian structures, and any (bi-)Hamiltonian structure can be viewed as a generalized one via the map \eqref{AQ}. Typically, we will use the notation $\diff{}{\qt}$ to denote generalized Hamiltonian structures to emphasize that they are vector fields on $J^\infty(\hat M)$. 

We take local coordinates $(u^1,\dots, u^n;\qth_1,\dots,\qth_n)$ of $\hat M$, then any generalized Hamiltonian structure of differential degree 1 takes the form 
\begin{equation}
    \label{AR}
  \diff{u^\qa}{\qt} = g^{\qa\qb}(u)\qth_\qb^1+\Qg^{\qa\qb}_\qg(u) u^\qg_x\qth_\qb,\quad \diff{\qth_\qa}{\qt} = V_\qa^{\mu\ql}(u)\qth_\ql\qth_\mu^1+Q^{\mu\ql}_{\qa\qg}(u)u^\qg_x\qth_\ql\qth_\mu,
\end{equation}
where we assume that 
\[
Q^{\ql\mu}_{\qa\qb}(u)+Q^{\mu\ql}_{\qa\qb}(u) = 0.
\]
We define a generalized Hamiltonian structure of hydrodynamic type to be that of the form \eqref{AR} with $\det(g^{\qa\qb}(u))\neq 0$. It turns out that the notion of generalized Hamiltonian structures of hydrodynamic type is characterized by the differential geometry of the manifold $\hat M_0$.
\begin{Th}[\cite{lorenzoni2026generalised}]
Let $\diff{}{\qt}$ be a generalized Hamiltonian structure of the form \eqref{AR}. Then the data 
\begin{equation}
    \label{AT}
\nabla_{\qp_\qa}\qp_\qb = A^\ql_{\qa\qb}\qp_\ql,\quad A^\ql_{\qa\qb} = -g_{\qb\mu}\qp_\qa g^{\ql\mu}+g_{\qb\qg}\Qg^{\ql\qg}_\qa,
\end{equation}
where $g_{\qa\qb}$ are elements of the inverse of $(g^{\qa\qb})$ given by $g^{\mu\qa}g_{\mu\qb} = g^{\qa\mu}g_{\qb\mu}= \qd^\qa_\qb$, defines a torsionless flat affine connection on $\hat M_0$. Conversely, any torsionless flat affine connection on $\hat M_0$ together with a choice of non-degenerate matrix $g^{\qa\qb}(u)$ gives rise to a generalized Hamiltonian structure of hydrodynamic type. 
\end{Th}

It is worth nothing that the affine connection $\nabla$ given by \eqref{AT} is globally well-defined on $\hat M_0$, meaning that it is independent of choice of local coordinates of $\hat M$. In contrast, the matrix $(g^{\qa\qb})$ does not give a well-defined geometric quantity on $\hat M_0$. Furthermore, the connection $\nabla$ is the only invariant defined by a generalized Hamiltonian structure of hydrodynamic type, meaning that for a fixed torsionless flat affine connection $\nabla$, different choices of the matrices $(g^{\qa\qb}(u))$ give rise to equivalent generalized Hamiltonian structures.

For a generalized bi-Hamiltonian structure $(\diff{}{\qt_0},\diff{}{\qt_1})$ taking the form 
\begin{align*}
      \diff{u^\qa}{\qt_0} &= g^{\qa\qb}\qth_\qb^1+\Qg^{\qa\qb}_\qg u^\qg_x\qth_\qb,\quad \diff{\qth_\qa}{\qt_0} = V_\qa^{\mu\ql}\qth_\ql\qth_\mu^1+Q^{\mu\ql}_{\qa\qg}u^\qg_x\qth_\ql\qth_\mu,\\
        \diff{u^\qa}{\qt_1} &= \tilde g^{\qa\qb}\qth_\qb^1+\tilde \Qg^{\qa\qb}_\qg u^\qg_x\qth_\qb,\quad \diff{\qth_\qa}{\qt_1} = \tilde V_\qa^{\mu\ql}\qth_\ql\qth_\mu^1+\tilde Q^{\mu\ql}_{\qa\qg}u^\qg_x\qth_\ql\qth_\mu,
\end{align*}
under local coordinates  $(u^1,\dots, u^n;\qth_1,\dots,\qth_n)$, the quantities
\[
L^\qb_\qa = \tilde g^{\qb\mu}g_{\qa\mu}
\]
define a global section $L$ of the vector bundle $\mathrm{End}(T\hat M_0)$ with vanishing Nijenhuis torsion. Together with the connections $\nabla$ and $\nabla^*$ corresponding to $\diff{}{\qt_0}$ and $\diff{}{\qt_1}$, the data $(L,\nabla,\nabla^*)$ completely determines $(\diff{}{\qt_0},\diff{}{\qt_1})$. We call the generalized bi-Hamiltonian structure $(\diff{}{\qt_0},\diff{}{\qt_1})$ to be semisimple if all the eigenvalues of $L$ are non-constant and pairwise distinct. Then it follows that there exist local coordinates $(u^1,\dots, u^n;\qth_1,\dots,\qth_n)$ of $\hat M$ such that 
\[
\diff{u^i}{\qt_0} = \qth_i^1+\sum_{j,k}A^i_{jk}u^k_x\qth_j,\quad \diff{u^i}{\qt_1} = u^i\qth_i^1+\sum_{j,k}B^i_{jk}u^k_x\qth_j,
\]
and such a coordinate system is called the canonical coordinates of $(\diff{}{\qt_0},\diff{}{\qt_1})$. 

\subsection{Formulation of the deformation problem}
In this section, we formulate the classification problem of the deformations of generalized (bi)-Hamiltonian structures of hydrodynamic types. We start by introducing natural coordinate transformations on $\hm A$. We fix local coordinates $(u^1,\dots, u^n;\qth_1,\dots,\qth_n)$ of $\hat M$.
\begin{Def}
    For any $f^1,\dots,f^n\in\hm A^0_{\geq 1}$ and $g_1,\dots,g_n\in\hm A^1_{\geq 1}$, the isomorphism of $\hm A$ given by  
    \begin{equation}
        \label{AV}
    u^{\qa,s}\mapsto u^{\qa,s}+\qp_x^s f^\qa,\quad \qth_\qa^s\mapsto\qth_\qa^s+\qp_x^s g_\qa,\quad s\geq 0,
    \end{equation}
    is called a generalized Miura-type trasnformation.
\end{Def}

Let us compare this definition to that of an ordinary Miura-type transformation. Recall that for a smooth manifold $M$ of dimension $n$, we can canonically construct the infinite jet bundle $J^\infty(\hat M)$. Then a Miura-type transformation (of the second kind) is an isomorphism on $\hm A$ of the form 
\[
 u^{\qa,s}\mapsto \qp_x^s \bar u^\qa,\quad \qth_\qa^s\mapsto \qp_x^s\sum_{t\geq 0}(-\qp_x)^t\diff{u^\qb}{\bar u^{\qa,t}}\qth_\qb
\]
where we have 
\[
\bar u^\qa = u^\qa+f^\qa,\quad f^\qa\in\hm A^0_{\geq 1}.
\]
Therefore, a Miura-type transformation is completely determined by the changes of even variables $u^\qa$, while a generalized Miura-type transformation allows additional freedom of changes of odd variables $\qth_\qa$. It is proved \cite{liu2011jacobi} that a Miura-type transformation preserves the Lie algebra structure \eqref{AU} defined on $\hm F$, hence a Miura-type transformation respects the natural symplectic structure on $\hat M$. On the other hand, a generalized Miura-type transformation is simply a general isomorphism of $\hm A$ as a differential graded algebra.

\begin{Rem}
    The notion of a generalized Miura-type transformation has been used in an essential way in \cite{liu2023variationalreci} to study the bi-Hamiltonian properties of linear reciprocal transformations.
\end{Rem}

\begin{Lem}[\cite{liu2023variationalreci}, Lemma 9]
For any generalized Miura-type transformation of the form \eqref{AV}, there exists a unique flow $T\in\derx^0_{\geq 1}$ such that the transformation reads 
\[
u^{\qa,s}\mapsto T(u^{\qa,s}),\quad \qth_\qa^s\mapsto T(\qth_\qa^s),\quad s\geq 0.
\]
Moreover, any flow $X\in\derx$ is transformed to $\exp(-\mathrm{ad}_T)X$ after performing the transformation \eqref{AV}.
\end{Lem}

Let us fix a generalized Hamiltonian structure $P\in\derx^1_1$ of hydrodynamic type. Then it follows that the classification of deformations of the generalized Hamiltonian structure $P$ is described by the cohomology groups
\[
H^p_d(\derx,P) = \frac{\derx^p_d\cap\ker \mathrm{ad}_P}{\derx^p_d\cap\ima \mathrm{ad}_P}.
\]
It is easy to see that the groups $H^1_{\geq 2}(\derx, P)$ classify the equivalence classes under generalized Miura-type transformations of infinitesimal deformation of $P$, and the groups  $H^2_{\geq 2d_0}(\derx, P)$ describe the obstructions for extending an infinitesimal deformation, where $d_0$ is the lowest differential degree of non-zero classes in $H^1_{\geq 2}(\derx, P)$. Similarly, for a generalized bi-Hamiltonian structure $(P_0,P_1)$ of hydrodynamic type, we define the groups 
\begin{equation}\label{BB}
BH^p_d(\derx,P_0,P_1) = \frac{\derx^p_d\cap\ker\mathrm{ad}_{P_0}\cap \ker\mathrm{ad}_{P_1}}{\derx^p_d\cap\ima \mathrm{ad}_{P_1}\comp \mathrm{ad}_{P_0}},
\end{equation}
then the groups $BH^1_{\geq 2}(\derx,P_0,P_1)$  classify the equivalence classes under generalized Miura-type transformations of infinitesimal deformation of $(P_0,P_1)$ and the groups  $BH^2_{\geq 2d_0}(\derx,P_0,P_1)$ describe the obstructions, where  $d_0$ is the lowest differential degree of non-zero classes in $BH^1_{\geq 2}(\derx,P_0,P_1)$.

\begin{Rem}
    When the generalized (bi)-Hamiltonian structure comes from a genuine (bi)-Hamiltonian structure via the map \eqref{AQ}, the corresponding cohomology groups coincide with the variational (bi-)Hamiltonian cohomology groups as introduced in \cite{liu2022variational,liu2023variational}.
\end{Rem}

\section{Classification of scalar semisimple generalized bi-Hamiltonian structure}
\label{AI}
\subsection{A Darboux theorem for generalized Hamiltonian structures}
Let us first prove that any deformation of a generalized Hamiltonian structure of hydrodynamic type is trivial. This generalizes the usual Darboux theorem for Hamiltonian structures of hydrodynamic type \cite{getzler2002darboux}. In what follows, we fix a generalized Hamiltonian structure $\diff{}{\qt}$ of hydrodynamic type.

\begin{Lem}
    \label{AS}
    For any constant non-degenerate matrix $\eta^{\qa\qb}$, there exist suitable  coordinates $(u^1,\dots,u^n,\qth_1,\dots,\qth_n)$ on $\hat M$ such that  
    \[
    \diff{u^\qa}{\qt} = \eta^{\qa\qb}\qth_\qb^1,\quad \diff{\qth_\qa}{\qt} = 0.
    \]
\end{Lem}
\begin{proof}
    By choosing flat coordinates of the connection $\nabla$ defined by $\diff{}{\qt}$, we may assume, after using the relations \eqref{AT}, that 
    \[
    \diff{u^\qa}{\qt} = g^{\qa\qb}(u)\qth_\qb^1+\qp_\qg g^{\qa\qb}(u)u^\qg_x\qth_\qb.
    \]
    For an arbitrary constant non-degenerate matrix $\eta^{\qa\qb}$, we set 
        \[
T^\qb_\qa(u) = g_{\ql\qa}(u)\eta^{\ql\qb}.
\]  
    Then after performing the change of coordinates 
    \[
    u^\qa \mapsto \bar u^\qa,\quad \qth_\qa\mapsto T^\qb_\qa(\bar u)\bar\qth_\qb,
    \]
    we arrive at 
    \[
    \diff{\bar u^\qa}{\qt} = \eta^{\qa\qb}\bar \qth_\qb^1.
    \]
Finally, it follows from Lemma 4.8 of \cite{lorenzoni2026generalised} that we must have 
\[
\diff{\bar\qth_\qa}{\qt} = 0.
\]
The lemma is proved.
\end{proof}

\begin{Th}
    \label{BM}
    For $p\geq -1$ and $d>0$, we have 
    \[H^p_d\kk{\derx,\diff{}{\qt}} = 0.\]
\end{Th}
\begin{proof}
    It follows from Lemma \ref{AS} that we can choose a  constant non-degenerate symmetric matrix $\eta^{\qa\qb}$ together with suitable local coordinates $(u^1,\dots,u^n,\qth_1,\dots,\qth_n)$ of $\hat M$ such that 
    \[
    \diff{u^\qa}{\qt} = \eta^{\qa\qb}\qth_\qb^1 = D_P(u^\qa),\quad \diff{\qth_\qa}{\qt} = 0 = D_P(\qth_\qa),
    \]
    where  
    \[
    P  =\frac12\int\eta^{\qa\qb}\qth_\qa\qth_\qb^1,
    \]
    and $D_P\in\derx^1_1$ is as defined in \eqref{AQ}.
    Pick a cocycle $X\in\derx^p_d$ with $X(u^\qa) = Y^\qa\in \hm A^p_d$ and $X(\qth_\qa) = Z_\qa\in\hm A^{p+1}_d$. Then it is easy to see that the cocycle condition reads 
    \begin{align}
        \label{AAF}
    \eta^{\qa\qb}\qp_x Z_\qb = (-1)^p\diff{Y^\qa}{\qt}.
    \end{align}
    If $p=-1$, we must have $Y^\qa = 0$ and hence $Z_\qa = 0$ since $d>0$. We conclude that 
    \[
H^{-1}_d\kk{\derx,\diff{}{\qt}} = 0,\quad d>0,
    \]
    and in what follows we assume $p\geq 0$. 

    Using the identity \eqref{AW}, the cocycle condition \eqref{AAF} implies that 
    \[
    0  = \int D_P(Y^\qa) = [P,\bar Y^\qa],\quad \bar Y^\qa = \int Y^\qa,
    \]
    and therefore it follows from the triviality of the Hamiltonian cohomology of $P$ that there exist differential polynomials $ Q^\qa\in\hm A^{p-1}_{d-1}$ such that 
    \[
    \bar Y^\qa = \fk{P}{\bar Q^\qa},\quad \bar Q^\qa = \int Q^\qa,
    \]
    in another word, there exist other differential polynomials $R_\qa\in\hm A^p_{d-1}$ such that 
    \[
    Y^\qa = \diff{Q^\qa}{\qt}+(-1)^p\eta^{\qa\qb}\qp_x R_\qb.
    \]
    Using the condition \eqref{AAF}, we arrive at 
    \[
    Z_\qa = \diff{R_\qa}{\qt},\quad Y^\qa = \diff{Q^\qa}{\qt}-(-1)^{p-1}\eta^{\qa\qb}\qp_x R_\qb,
    \]
    from which we conclude that 
    \[
    X  =\fk{\diff{}{\qt}}{O},
    \]
    where $O\in\derx^{p-1}_{d-1}$ is defined by 
    \[
    O(u^\qa) = Q^\qa,\quad O(\qth_\qa) = R_\qa.
    \]
    Hence, the cocycle $X$ is also a coboundary and the theorem follows.
\end{proof}

\begin{Rem}
    The above theorem is equivalent to the triviality of the variational cohomology groups for Hamiltonian structures of hydrodynamic type \cite{liu2023variational}.
\end{Rem}

\begin{Cor}
    \label{BX}
    Any deformation of $\diff{}{\qt}$ is trivial, namely any deformation of $\diff{}{\qt}$ can be eliminated by performing a generalized Miura-type transformation.
\end{Cor}

\subsection{Classification of scalar semisimple generalized bi-Hamiltonian structure}
\label{BY}
In what follows, we will concentrate on the scalar semisimple generalized bi-Hamiltonian structure. By taking the corresponding canonical coordinates $(u,\qth)$, any scalar semisimple generalized bi-Hamiltonian structure takes the form 
\begin{align}
    \label{AZ}
\diff{u}{\qt_0} &=\qth^1+f(u)u_x\qth,\quad \diff{\qth}{\qt_0} = f(u)\qth\qth^1,\\
\label{BA}
\diff{u}{\qt_1} &=u\qth^1+\kk{1+ug(u)}u_x\qth,\quad \diff{\qth}{\qt_1} = \kk{1+ug(u)}\qth\qth^1,
\end{align}
for some smooth functions $f(u)$ and $g(u)$. Let us define the function
\begin{equation}
    \label{BC}
\bar K(u) = 1+u\left(g(u)-f(u)\right).
\end{equation}
Note that this function is well-defined since the canonical coordinates are unique.
\begin{Lem}
    \label{BH}
    If $\bar K(u) = \qa$ is a constant function, then there exist local coordinates $(v,\qs)$ such that 
\begin{align}
    \label{AX}
\diff{v}{\qt_0} &=\qs^1,\quad \diff{\qs}{\qt_0} = 0,\\
\label{AY}
\diff{v}{\qt_1} &=\tilde f(v)\qs^1+\qa \tilde f'(v)v_x\qs,\quad \diff{\qs}{\qt_1} = \qa \tilde f'(v)\qs\qs^1,
\end{align}
for some non-constant smooth function $\tilde f(v)$. In particular, when $\qa = \frac12$, we have 
\[
\diff{}{\qt_0} = D_{P_0},\quad \diff{}{\qt_1} = D_{P_1},
\]
 where 
    \[
    P_0 = \frac12\int \qs\qs^1,\quad P_1 = \frac12\int \tilde f(v)\qs\qs^1.
    \]
\end{Lem}
\begin{proof}
    It suffices to compute the function $\bar K(u)$ of the generalized bi-Hamiltonian structure \eqref{AX},\eqref{AY} in terms of the canonical coordinates. Take the local coordinates
    \[
    u = \tilde f(v),\quad \qth = \tilde f'(u)\qs,
    \]
    then a straightforward computation implies that 
    \[
    \diff{u}{\qt_0} = \qth^1-\frac{\tilde f''(v)}{\tilde f'(v)^2}u_x\qth,\quad \diff{u}{\qt_1} = u\qth^1+\kk{\qa-\frac{u\tilde f''(v)}{\tilde f'(v)^2}}u_x\qth,
    \]
    from which it is easy to see that $\bar K(u)=\qa$. The lemma is proved.
\end{proof}

Finally, we classify the scalar semisimple generalized bi-Hamiltonian structures using the function $\bar K(u)$. The classification will be justified in the next section, after we study the deformation problem.
\begin{Def}
    Let $(\diff{}{\qt_0},\diff{}{\qt_1})$ be a scalar semisimple generalized bi-Hamiltonian structure of the form \eqref{AZ}, \eqref{BA}. We call $(\diff{}{\qt_0},\diff{}{\qt_1})$ is 
    \begin{enumerate}
        \item of rigid type if $\bar K(u)$ is a non-constant function, or it is a constant function $\bar K(u) = \qa$ with $\qa\in\mathbb R\backslash \mathbb{Q}$;
        \item of $D$-type if $\bar K(u) = \frac 12$;
        \item of viscous type if $\bar K(u) = 1$;
        \item of exceptional type if otherwise.
    \end{enumerate}
\end{Def}

\begin{Rem}
    The term $D$-type is used here in a slightly more general meaning than that is used in \cite{liu2023variationalreci}.
\end{Rem}

\begin{Cor}
    \label{BG}
    If the scalar semisimple generalized bi-Hamiltonian structure $(\diff{}{\qt_0},\diff{}{\qt_1})$ is of $D$-type, then its deformation is parametrized by a single function called the central invariant. In particular, any scalar generalized bi-Hamiltonian structure with $D$-type leading terms arises from a genuine bi-Hamiltonian via the map \eqref{AQ}.
\end{Cor}
\begin{proof}
    It follows from Lemma \ref{BH} that, after choosing suitable local coordinates, we have 
    \[
    \diff{}{\qt_0} = D_{P_0},\quad \diff{}{\qt_1} = D_{P_1}
    \]
    for some scalar semisimple bi-Hamiltonian structure $(P_0,P_1)$. It follows that
    \[
     BH^p_d\left(\derx,\diff{}{\qt_0},\diff{}{\qt_1}
    \right) = VBH^p_d\left(\derx,P_0,P_1\right),
    \]
    where the right-hand side is called the variational bi-Hamiltonian cohomologies as introduced and studied in \cite{liu2023variational}. Then the corollary follows from the results presented in \cite{liu2023variational,liu2023variationalreci}.
\end{proof}

\subsection{Deformation of scalar semisimple generalized bi-Hamiltonian structure}
In this section, we prove Theorem \ref{AG} and Theorem \ref{AH} by computing corresponding cohomology groups \eqref{BB}. In what follows, we will use the following fact repeatedly.
\begin{Lem}[\cite{liu2013bihamiltonian}, Lemma 4.4]
    \label{BZ}
    Let $(\derx,\qp_0,\qp_1)$ be a cochain bicomplex with the differentials 
    \[
    \qp_0,\qp_1\colon \derx^p_d\to\derx^{p+1}_{d+1}
    \]
    such that $H^p_{> 0}(\derx,\qp_0) = 0$ for any $p\geq -1$. Then 
    \[
    \derx[\ql] = \derx\otimes\mathbb R[\ql],\quad \qp_\ql = \qp_1-\ql\qp_0
    \]
    defines a cochain complex with 
    \[
    BH^p_d\left(\derx,\qp_0,\qp_1
    \right)\cong H^p_d\left(\derx[\ql],\qp_\ql\right),\quad d\geq 2
    \]
    where $BH^p_d(\derx,\qp_0,\qp_1)$ is similarly defined as in \eqref{BB}, and the isomorphism above is induced from the natural embedding 
    \[
    \derx\to\derx[\ql].
    \]
    A similar result hold true if we replace the space $\derx$ by $\hm A$.
\end{Lem}

Let us fix a scalar semisimple generalized bi-Hamiltonian structure of the form \eqref{AZ}, \eqref{BA}. Then it follows that 
    \[
    BH^p_d\left(\derx,\diff{}{\qt_0},\diff{}{\qt_1}
    \right)\cong H^p_d\left(\derx[\ql],\diff{}{\qt(\ql)}\right),\quad d\geq 2,
    \]
    with \[
    \diff{}{\qt(\ql)} = \diff{}{\qt_1}-\ql\diff{}{\qt_0}.
    \]
    Let $X\in\derx^p_d$ be a flow whose action is determined by 
    \[
    X(u) = P,\quad X(\qth) = Q,\quad P\in\hm A^p_d,\quad Q\in\hm A^{p+1}_d.
    \]
    After a straightforward computation, we arrive at 
    \begin{align*}
        \fk{\diff{}{\qt(\ql)}}{X}u =&\,\diff{P}{\qt(\ql)}-\qth^1P-(-1)^p(u-\ql)\qp_x Q-\diff{K(u,\ql)}{u}u_x\qth P\\[1em]
        &-K(u,\ql)\qth\qp_x P-(-1)^p K(u,\ql)u_xQ,\\[1em]
        \fk{\diff{}{\qt(\ql)}}{X}\qth =&\, \diff{Q}{\qt(\ql)}-(-1)^p \diff{K(u,\ql)}{u}\qth\qth^1P+K(u,\ql)\qth^1 Q-K(u,\ql)\qth\qp_x Q,
    \end{align*}
where $K(u,\ql) = 1+ug(u)-\ql f(u)$ and 
\[
\diff{}{\qt(\ql)} = \sum_{s\geq 0}\qp_x^s\kk{(u-\ql)\qth^1+K(u,\ql)u_x\qth}\diff{}{u^{(s)}}+\qp_x^s\kk{K(u,\ql)\qth\qth^1}\diff{}{\qth^s}.
\]
Therefore, let us denote by 
\[
\hm B^p_d = \hm A^p_d[\ql]\oplus\hm A^{p+1}_d[\ql],\quad \hm B= \bigoplus_{p,d}\hm B^p_d[\ql],
\]
then it follows from the above computation that 
 \[
   H^p_d\left(\derx[\ql],\diff{}{\qt(\ql)}\right)\cong H^p_d(\hm B,D),
    \]
    where the action of differential $D$ on $\hm B^p_d$ is given by 
    \begin{align*}
        D &= \begin{pmatrix}D_1 & D_2\\D_3 & D_4\end{pmatrix},\\
            D_1&=\diff{}{\qt(\ql)}-K\qth\qp_x -\kk{\qth^1+\diff{K}{u}u_x\qth},\\
    D_2&=(-1)^{p+1}(u-\ql)\qp_x+(-1)^{p+1}Ku_x,\\
    D_3&=(-1)^{p+1}\diff{K}{u}\qth\qth^1,\\
    D_4&=\diff{}{\qt(\ql)}-K\qth\qp_x+K\qth^1.
    \end{align*}
    
In what follows, we compute the cohomology groups $H^p_d(\hm B,D)$ using a similar method that is used in \cite{carlet2018deformations,carlet2018central,liu2023variational}. Firstly, we endow the complex $\hm B$ with a filtration as follows. A monomial in $\hm A$ is said to have length $\ell$ if it is of the form 
\[
h(u)u^{(k_1)}\dots u^{(k_\ell)}\qth^{i_1}\dots\qth^{i_p},\quad k_1,\dots,k_\ell \geq 1.
\]
Define the subspace $F^k\hm A^p\subseteq\hm A^p$ consisting of differential polynomials all of whose monomials have length $\ell\geq k-p$. Denote by
\[
 F^k\hm B^p =  F^k\hm A^p[\ql]\oplus F^k\hm A^{p+1}[\ql],
\]
then it is not hard to see that 
\[
\hm B^p = F^0\hm B^p\supseteq F^1\hm B^p\supseteq\dots \supseteq F^k\hm B^p\supseteq F^{k+1}\hm B^p\supseteq\dots,
\]
and each $F^k\hm B^\bullet$ is a subcomplex of $\hm B$. We denote by $(E_r,\Qd_r)_{r\geq 0}$ the corresponding spectral sequence.

We start by computing the first page $E_1 = H(E_0,\Qd_0)$ of the spectral sequence. It follows that $E_0 = \hm B$ and the action of $\Qd_0$ on $E_0^p$ reads
\[
\Qd_0 = \begin{pmatrix}(u-\ql)\hat d & (-1)^{p+1}(u-\ql)\bar\qp_x \\
    0 & (u-\ql)\hat d
\end{pmatrix},
\] 
where we denote 
\[
\hat d = \sum_{s\geq 1}\qth^{s+1}\diff{}{u^{(s)}},\quad \bar\qp_x = \qp_x-u_x\diff{}{u}.
\]
We can compute the cohomology of $E_0$ using the technique of mapping cones. Indeed, the complex $(E_0,d_0)$ is the same as the mapping cone of the cochain map 
\[
\rho\colon \left(\hm A^\bullet[\ql],(u-\ql)\hat d\right)\to \left(\hm A^\bullet[\ql],(u-\ql)\hat d\right), \quad \rho(\qo) = (-1)^p (u-\ql)\bar\qp_x\qo,\quad \qo\in\hm A^p[\ql].
\]
Therefore, a standard fact in homological algebra implies the existence of the short exact sequence 
    \xym{
    0\ar[r] & \cok\rho_*\ar[r]^{i_*} & H(E_0,\Qd_0)\ar[r]^{\pi_*} & \ker\rho_*\ar[r] & 0,
}
where $\rho_*$ is the map on cohomology groups induced by $\rho$, $i_*$ is the inclusion to the first component of $E_0$ and $\pi_*$ is the projection of the second component of $E_0$.

\begin{Lem}[\cite{carlet2018deformations}, Proposition 11]
    Denote by 
    \begin{equation}
        \label{BF}
    \hm C = C^\infty(u)[\qth,\qth^1],
    \end{equation}
    then we have 
    \[H\left(\hm A[\ql],(u-\ql)\hat d\right) = \hm C[\ql]\oplus\frac{\hat d(\hm A)[\ql]}{(u-\ql)\hat d(\hm A)[\ql]}.\]
\end{Lem}
Let us proceed to compute the kernel and the cokernel of the map $\rho_*$. Take an arbitrary class $[\qo]\in H\left(\hm A[\ql],(u-\ql)\hat d\right)$ which is represented by 
\[
\qo = \varphi_1+\varphi_2\qth+\varphi_3\qth^1+\varphi_4\qth\qth^1+\hat d(\qa),\quad \qa\in\hm A[\ql],
\]
where functions $\varphi_i\in C^\infty(u)[\ql]$, then it follows that
\[
\rho_*[\qo] = (u-\ql)\kk{\varphi_2\qth^1+\varphi_3\qth^2+\varphi_4\qth\qth^2},
\]
here we use the obvious fact that $[\bar \qp_x,\hat d] = 0$ and hence the class $(u-\ql)\bar\qp_x\hat d(\qa)$ vanishes in the cohomology group. Furthermore, note that 
\[
(u-\ql)\varphi_3\qth^2 = (u-\ql)\varphi_3\hat d(u_x),\quad (u-\ql)\varphi_4\qth\qth^2 = -(u-\ql)\varphi_4\hat d(u_x\qth),
\]
from which we conclude that 
\[
\rho_*[\qo] = (u-\ql)\varphi_2\qth^1.
\]
Define the map 
\[
\phi\colon \hm C[\ql]\to \hm C[\ql],\quad \varphi_1+\varphi_2\qth+\varphi_3\qth^1+\varphi_4\qth\qth^1\mapsto (u-\ql)\varphi_2\qth^1.
\]
Summarizing the computations above, we arrive at the following result.
\begin{Th}
    There exists a spilt exact sequence 
        \xym{
    0\ar[r] & \cok\rho_*\ar[r]^{i_*} & H(E_0,\Qd_0)\ar[r]^{\pi_*} & \ker\rho_*\ar[r]\ar@/^1pc/[l]^s & 0,
}
where 
\[
\cok\rho_* = \cok\phi\oplus \frac{\hat d(\hm A)[\ql]}{(u-\ql)\hat d(\hm A)[\ql]},\quad \ker\rho_* = \ker\phi\oplus \frac{\hat d(\hm A)[\ql]}{(u-\ql)\hat d(\hm A)[\ql]},
\]
and the splitting morphism reads 
\[
s\kk{\varphi_1+\varphi_3\qth^1+\varphi_4\qth\qth^1+\hat d(\qa)} = \begin{pmatrix}(-1)^p\kk{\varphi_3u_x-\varphi_4u_x\qth+\bar\qp_x\qa}\\ \varphi_1+\varphi_3\qth^1+\varphi_4\qth\qth^1+\hat d(\qa)\end{pmatrix},
\]
here $p+1$ is the super degree of the class $\varphi_1+\varphi_3\qth^1+\varphi_4\qth\qth^1+\hat d(\qa)$.
\end{Th}

We proceed to consider the cohomology groups $H(E_1,\Qd_1)$. It is easy to see that the action of differential of $\Qd_1$ on classes of super degree $p$ is given by 
\[
\begin{pmatrix}\qd-K\qth\bar\qp_x-\qth^1 & (-1)^{p+1}(u-\ql)u_x\diff{}{u}+(-1)^{p+1}Ku_x\\[1em]0 & \qd-K\qth\bar\qp_x+K\qth^1\end{pmatrix},
\]
where $\qd$ is the vector field given by 
\[
\qd = (u-\ql)\qth^1\diff{}{u}+\sum_{s\geq 1}\sum_{j=1}^{s}\binom{s}{j}u^{(j)}\qth^{s-j+1}\diff{}{u^{(s)}}+\sum_{s\geq 1}K\qp_x^s(u_x\qth)\diff{}{u^{(s)}}+\sum_{s\geq 0}K\qp_x^s\kk{\qth\qth^1}\diff{}{\qth^s}.
\]

To consider the cohomology $H(E_1,\Qd_1)$, we first make some simplifications. Note that the components $\cok\phi$ and $\ker\phi$ of $\cok\rho_*$ and $\ker\rho_*$ contribute only classes of differential degree $d\leq 1$, therefore we can ignore them since we are only interested in cohomologies for differential degree $d\geq 2$. In particular, we have the isomorphism
\begin{equation*}
    \label{BE}
\hat d(\hm A^{p-1}_{d-1})\oplus \hat d(\hm A^{p}_{d-1})\to H^p_d(E_0,\Qd_0),\quad \begin{pmatrix}
    \hat d(\qa)\\\hat d(\qb)
\end{pmatrix}\mapsto \begin{pmatrix}
    \hat d(\qa)+(-1)^p\bar\qp_x\qb\\\hat d(\qb)
\end{pmatrix},\quad d\geq 2,
\end{equation*}
where we also identify the space $\hat d(\hm A)[\ql]/(u-\ql)\hat d(\hm A)[\ql]$ with $\hat d(\hm A)$ by setting $\ql = u$. Next, we introduce another gradation by setting 
\[
\deg_{\qth^1} u^{(s)} = 0,\quad \deg_{\qth^1}\qth^s = \qd_{s,1},\quad s\geq 0,
\]
and it follows that such a gradation is also well-defined on $E_1$. We observe that the differential $\Qd_1$ can be decomposed into three parts, namely the components that decrease, preserve or increase $\deg_{\qth^1}$, and we define a filtration on $E_1$ using the gradation $\deg_{\qth^1}-\deg_{\qth}$. This induces a second spectral sequence $E'$ whose zero page is $E'_0 = E_1$, and the differential on $E'_0$ is given by the components of $\Qd_1$ that increase $\deg_{\qth^1}$ which reads 
\begin{align*}
\Qd' &= \begin{pmatrix}\qth^1\mathcal E-\qth^1 & 0\\ 0 & \qth^1\mathcal E+\qth^1\bar K(u)\end{pmatrix},\\
 \mathcal E &= \sum_{s\geq 1}\kk{1+s\bar K(u)}u^{(s)}\diff{}{u^{(s)}}+\sum_{s\geq 2}(s-1)\bar K(u)\qth^s\diff{}{\qth^s},
\end{align*}
where we have identified $\ql$ with $u$, and 
\[
\bar K(u) = K(u,u) = 1+u(g(u)-f(u))
\]
as defined in \eqref{BC}.
\begin{Prop}
    \label{BJ}
    If the generalized bi-Hamiltonian structure is of rigid type or of viscous type, then we have
    \begin{equation*}
        \label{BD}
    H^p_d\left(\hat d(\hm A),\qth^1\mathcal E+\qth^1\bar K(u)\right) = 0,\quad p\geq -1,\quad d\geq 2.
    \end{equation*}
\end{Prop}
\begin{proof}
    First observe that the ring of differential polynomials can also be written as 
    \[
    \hm A = \hm C[[u^{(s)},\qth^{s+1}\colon s\geq 1]],
    \]
    where $\hm C$ is as defined in \eqref{BF}. Let us take a monomial 
\[
\mathfrak{m} = u^{(s_1)}\dots u^{(s_n)}\qth^{k_1+1}\dots\qth^{k_m+1},\quad s_i,\,k_j\geq 1,
\]
and a general element 
\[
\varphi = \varphi_1(u)+\varphi_2(u)\qth+\varphi_3(u)\qth^1+\varphi_4(u)\qth\qth^1\in\hm C.
\]
Using the fact that 
\[
\fk{\mathcal E}{\hat d} = -\hat d,
\]
it follows from a straightforward computation that 
\[
\left(\qth^1\mathcal E+\qth^1\bar K(u)\right)(\hat d(\varphi \mathfrak{m})) = \qth^1 w(\mathfrak{m})\hat d(\varphi\mathfrak{m}),
\]
where the weight $w(\fm)$ is given by 
\[
w(\fm) = n-1+\kk{\sum_{i=1}^n s_i+\sum_{j=1}^m k_j +1}\bar K(u).
\]
Furthermore, we have 
\[
\qth^1\hat d(\varphi\mathfrak{m}) = \varphi_1\qth^1\hat d(\fm)+\varphi_2\qth\qth^1\hat d(\fm),
\]
and hence we conclude that 
\[
\hat d(\hm A) = \bigoplus_\fm \hm C\hat d(\fm)
\]
decomposes the complex $\hat d(\hm A)$ into subcomplexes. 

Let us concentrate on the cohomology group \[H\left(\hm C\hat d(\fm),\qth^1\mathcal E+\qth^1\bar K(u)\right)\] for some fixed monomial $\fm$. If $w(\fm)$ is not identically zero, we conclude that a cocycle $\hat d(\varphi\fm)$ must satisfy 
\[
\varphi_1 = \varphi_2 = 0.
\]
However, the above computation implies that 
\[\hat d\kk{\varphi_3\qth^1\fm+\varphi_4\qth\qth^1\fm} = \kk{\qth^1\mathcal E+\qth^1\bar K}\hat d\kk{\frac{-\varphi_3\fm+\varphi_4\qth\fm}{w(\fm)}},\]
from which it follows that 
 \[H\left(\hm C\hat d(\fm),\qth^1\mathcal E+\qth^1\bar K(u)\right) = 0\]
 whenever $w(\fm)$ is not identically zero. In particular, when the generalized bi-Hamiltonian is of rigid type, namely when $\bar K(u)$ is not a constant function or an irrational constant, the weight $w(\fm)$ is always non-zero. For $\bar K(u) =1$, it is also obvious that all non-zero monomials have non-zero weights. The proposition is proved.
\end{proof}
\begin{Prop}
    \label{BK}
    If the generalized bi-Hamiltonian structure is of rigid type then we have
    \begin{equation}
        \label{BI}
    H^p_d\left(\hat d(\hm A),\qth^1\mathcal E-\qth^1\right) = 0,\quad p\geq -1,\quad d\geq 2.
    \end{equation}
    If the generalized bi-Hamiltonian structure is of viscous type, the cohomology groups \eqref{BI} vanish for all $p\geq -1, d\geq 2$ except for 
    \[
    (p,d) = (1,2),\ (2,2),\ (2,3),\ (3,3).
    \]
\end{Prop}
\begin{proof}
    The proof is exactly the same as that of Proposition \ref{BJ}, and the only difference is that in this case the weight function $w(\fm)$ of a monomial 
    \[
    \mathfrak{m} = u^{(s_1)}\dots u^{(s_n)}\qth^{k_1+1}\dots\qth^{k_m+1},\quad s_i,\,k_j\geq 1,
    \]
    reads 
    \[
    w(\fm) = n-2+\kk{\sum_{i=1}^n s_i+\sum_{j=1}^m k_j}\bar K(u).
    \]
    This function is again non-zero for generalized bi-Hamiltonian structures of rigid type, and this proves the first assertion. It only remains to consider the viscous case $\bar K(u) = 1$. In this case, the only non-zero cocycle with a vanishing weight is given by 
    \[
 \kk{\varphi_1(u)+\varphi_2(u)\qth+\varphi_3(u)\qth^1+\varphi_4(u)\qth\qth^1}\hat d(u_x),
    \] 
    and therefore the non-zero cohomology groups are only possible at degrees $(p,d) = (1,2),(2,2),(2,3)$ or $(3,3)$. The proposition is proved. 
\end{proof}

\begin{Th}
    \label{BO}
    If the generalized bi-Hamiltonian structure is of rigid type then we have
    \begin{equation}
        \label{BL}
       BH^p_d\left(\derx,\diff{}{\qt_0},\diff{}{\qt_1}
    \right) = 0,\quad p\geq -1,\quad d\geq 2.
    \end{equation}
    If the generalized bi-Hamiltonian structure is of viscous type, then the above cohomology groups vanish for all $p\geq -1, d\geq 2$ except for 
    \[
    (p,d) = (1,2),\ (2,2),\ (2,3),\ (3,3).
    \]
\end{Th}
\begin{proof}
    For generalized bi-Hamiltonian structure of rigid type, it follows from  Proposition \ref{BJ} and Proposition \ref{BK} that the spectral sequence $E'$ collapses on the first page with $E'_1 = E'_\infty = 0$. Since $E'\Rightarrow E_1$,  we conclude that the spectral sequence $E$ also collapses on its first page with $E_1 = E_\infty = 0$, which implies the vanishing result \eqref{BL}. The case of viscous type follows from a same argument.
\end{proof}

\begin{Prop}
    \label{BP}
   If the generalized bi-Hamiltonian structure is of viscous type, we have 
    \[
    BH^1_2\left(\derx,\diff{}{\qt_0},\diff{}{\qt_1}
    \right) \cong C^\infty(u),
    \]
    and any class can be uniquely represented by 
    \begin{equation}
        \label{BN}
X = \fk{\diff{}{\qt_1}}{\fk{\diff{}{\qt_0}}{A}}\in\derx^1_2,
    \end{equation}
    where $A$ is a flow commuting with $\qp_x$ of super degree -1, differential degree 0 given by 
    \[
    A(u) = 0,\quad A(\qth) = a(u)\log(u_x).
    \]
    The function $a(u)$ is called the viscous central invariant of the class $X$.
\end{Prop}
\begin{proof}
    Let us fix a scalar generalized bi-Hamiltonian structure of viscous type which is given by 
\begin{align*}
\diff{u}{\qt_0} &=\qth^1+f(u)u_x\qth,\quad \diff{\qth}{\qt_0} = f(u)\qth\qth^1,\\
\diff{u}{\qt_1} &=u\qth^1+\kk{1+uf(u)}u_x\qth,\quad \diff{\qth}{\qt_1} = \kk{1+uf(u)}\qth\qth^1.
\end{align*}
    By a straightforward computation, we see that 
    \[
    \derx^1_2\cap\ima \mathrm{ad}_{\diff{}{\qt_1}}\comp \mathrm{ad}_{\diff{}{\qt_0}} = 0, 
    \]
    and hence we have 
    \[
    BH^1_2\left(\derx,\diff{}{\qt_0},\diff{}{\qt_1}
    \right) = \derx^1_2\cap\ker \mathrm{ad}_{\diff{}{\qt_0}}\cap \ker \mathrm{ad}_{\diff{}{\qt_1}}.
    \]
    Fix a class $X$ in the cohomology group, then it follows from Theorem \ref{BM} that there exists flows $Y_0,Y_1\in\derx^0_1$ such that 
    \[
    X = \fk{\diff{}{\qt_0}}{Y_0}= \fk{\diff{}{\qt_1}}{Y_1}.
    \] 
    Note that such $Y_i$ is not unique, and has the freedom by adding elements in $\derx^0_1\cap\ima \mathrm{ad}_{\diff{}{\qt_i}}$. Indeed, for $H\in\derx^{-1}_0$ given by 
    \[
    H(u) = 0,\quad H(\qth) = h(u),
    \]
    we have 
\begin{align*}
    \fk{\diff{}{\qt_0}}{H}u &= (fh+h')u_x,\quad \fk{\diff{}{\qt_0}}{M}\qth = (fh+h')\qth^1,\\
    \fk{\diff{}{\qt_1}}{H}u &= (h+ufh+uh')u_x,\quad \fk{\diff{}{\qt_1}}{H}\qth = (h+ufh+uh')\qth^1.
\end{align*}
Therefore, we can take $Y_0,Y_1$ to be the flows given by 
\[
Y_i(u) = 0,\quad Y_i(\qth) = p_i(u)\qth^1+q_i(u)u_x\qth,\quad i = 0,1.
\]
Then we have 
\begin{align*}
        \fk{\diff{}{\qt_0}}{Y_0}u &=-p_0\qth^2-(fp_0+q_0+p_0')u_x\qth^1-q_0u_{xx}\qth-(fq_0+q_0')u_x^2\qth,\\
     \fk{\diff{}{\qt_0}}{Y_0}\qth  & = -q_0\qth\qth^2+(p_0f'-q_0'-2fq_0)u_x\qth\qth^1,\\
     \fk{\diff{}{\qt_1}}{Y_1}u &=-up_1\qth^2-\left(p_1+ufp_1+uq_1+up_1'\right)u_x\qth^1-uq_1u_{xx}\qth-(q_1+ufq_1+uq_1')u_x^2\qth,\\
     \fk{\diff{}{\qt_0}}{Y_1}\qth  & = -uq_1\qth\qth^2+(up_1f'-uq_1'-2ufq_1+fp_1-3q_1)u_x\qth\qth^1.
\end{align*}
It follows that $p_0 = up_1$ and $q_0 = uq_0$. But we also need 
\[
p_0f'-q_0'-2fq_0 = up_1f'-uq_1'-2ufq_1+fp_1-3q_1,
\]
which gives 
\[q_1(u) = \frac{1}{2}fp_1.\]
Finally, the identity \eqref{BN} can be verified directly with 
\[
a(u) = -\frac12 p_0,
\]
and $p_0(u)$ serves as a free functional parameter in the cohomology group. 
The proposition is proved.
\end{proof}
\begin{Rem}
    Note that the identity \eqref{BN} does not imply that $X$ is a trivial class, since the flow $A$ does not take values in $\hm A$. Instead, \eqref{BN} shows that any class is quasi-trivial, which resembles the quasi-triviality for deformations of ordinary semisimple bi-Hamiltonian structures.
\end{Rem}

Summarizing the results of Corollary \ref{BG}, Theorem \ref{BO} and Proposition \ref{BP}, we conclude that Theorem \ref{AG} holds true.

\begin{Th}
    For any scalar generalized bi-Hamiltonian structure of the form \eqref{AZ}, \eqref{BA}, we have 
    \[
       BH^0_d\left(\derx,\diff{}{\qt_0},\diff{}{\qt_1}
    \right) = 0,\quad d\geq 2.
    \]
    In particular, Theorem \ref{AH} holds true.
\end{Th}
\begin{proof}
    First note that 
    \[
    BH^0_d\left(\derx,\diff{}{\qt_0},\diff{}{\qt_1}
    \right) = \derx^0_d\cap\ker \mathrm{ad}_{\diff{}{\qt_0}}\cap \ker \mathrm{ad}_{\diff{}{\qt_1}}.
    \]
    Let us fix a class $X\in BH^0_d\left(\derx,\diff{}{\qt_0},\diff{}{\qt_1}
    \right)$. Using Theorem \ref{BM}, it follows that there exist unique flows $P,Q\in\derx^{-1}_{d-1}$ such that 
    \[
    X = \fk{\diff{}{\qt_0}}{P} = \fk{\diff{}{\qt_1}}{Q}, 
    \]
    where we set 
    \[
    P(u)=Q(u) = 0,\quad P(\qth) = p,\quad Q(\qth) = q,\quad p,q\in\hm A^0_{d-1}.
    \]
    We denote 
    \begin{align*}
        W &= \fk{\diff{}{\qt_0}}{P}u - \fk{\diff{}{\qt_1}}{Q}u\\
         &= (\qp_xp+f u_x p)-(u\qp_xq+u_xq+ugu_xq),\\[1em]
        Z &= \fk{\diff{}{\qt_0}}{P}\qth - \fk{\diff{}{\qt_1}}{Q}\qth\\
         &= \kk{\diff{p}{\qt_0}+fp\qth^1-f\qp_xp\qth}-\kk{\diff{q}{\qt_1}+(1+ug)q\qth^1-(1+ug)\qp_xq\qth}.
    \end{align*}
    Assume that there exists some $N\geq 2$ such that 
    \[
    \diff{p}{u^{(s)}} = \diff{q}{u^{(s)}}=0,\quad \forall\,s\geq N+1.
    \]
    Then the identity 
    \[
    \diff{W}{u^{(N+1)}} = 0
    \]
    implies that 
    \[
    \diff{p}{u^{(N)}} = u \diff{q}{u^{(N)}}.
    \]
    And the identity 
    \[
    \diff{Z}{\qth^N} = 0
    \]
    gives the relation 
    \[
    \diff{p}{u^{(N-1)}} = \left(N+\bar{K}(u)\right)u_x\diff{q}{u^{(N)}}+u\diff{q}{u^{(N-1)}},
    \]
    while from the identity 
    \[\diff{W}{u^{(N)}} = 0\]
    it follows that 
    \[
     \diff{p}{u^{(N-1)}} = \left(\bar{K}(u)-1\right)u_x\diff{q}{u^{(N)}}+u\diff{q}{u^{(N-1)}}.
    \]
    Therefore, we conclude that 
    \[
    \diff{p}{u^{(N)}} = \diff{q}{u^{(N)}} = 0,
    \]
    and by induction, we must have
    \[
    p = ur(u)u_x^{d-1},\quad q = r(u)u_x^{d-1}
    \] 
    for some function $r(u)$. Finally, we can verify that $r(u) = 0$ by a straightforward calculation. The theorem is proved.
\end{proof}

\section{Examples}
\label{AJ}
In this section, we provide some examples and applications of the theory of generalized bi-Hamiltonian structures.

\subsection{Burgers hierarchy as a reduction of 1-constrained KP hierarchy}
\label{AD}
In this section, we provide a construction of the generalized bi-Hamiltonian structure of the Burgers hierarchy from the view point of reduction of a bi-Hamiltonian integrable hierarchy.

Recall that the 1-constrained KP hierarchy is described by the Lax formalism 
\[
\diff{L}{t_k} = \fk{L^k_+}{L},\quad L = \qp_x+(\qp_x-u)^{-1}w,\quad k\geq 1,
\]
where the pseudo-differential operator 
\begin{equation}
    \label{BR}
(\qp_x-u)^{-1} = \sum_{s\geq 1}g_s\qp_x^{-s},\quad g_s = (u-\qp_x)^{s-1}1.
\end{equation}
For example, the first three flows in the hierarchy read
\begin{align*}
&\begin{cases}
\diff{w}{t_1}= w_x,\\[1em]
\diff{u}{t_1}= u_x,
\end{cases}\quad
\begin{cases}
\diff{w}{t_2}= 2uw_x+2wu_x-w_{xx},\\[1em]
\diff{u}{t_2}= 2w_x+2uu_x+u_{xx},
\end{cases}\\[1em]
&\begin{cases}
\diff{w}{t_3}= 6ww_x+6uwu_x+3u^2w_x-3uw_{xx}-3u_xw_x+w_{xxx},\\[1em]
\diff{w}{t_3}= 6wu_x+6uw_x+3u^2u_x+3u_x^2+3uu_{xx}+u_{xxx}.
\end{cases}
\end{align*}

It is well-known that constrained KP Hierarchies possess bi-Hamiltonian structures \cite{cheng1995modifying,oevel1993constrained}. In our case, the flows in the 1-constrained KP hierarchy admit the bi-Hamiltonian formalism given by 
\begin{equation}
    \label{BU}
\begin{pmatrix}\diff{w}{t_k}\\[1em]\diff{u}{t_k}\end{pmatrix} = \mathcal P_0 \begin{pmatrix}\vard{H_k}{w}\\[1em]\vard{H_k}{u}\end{pmatrix} = \mathcal P_1 \begin{pmatrix}\vard{H_{k-1}}{w}\\[1em]\vard{H_{k-1}}{u}\end{pmatrix},\quad k\geq 1,
\end{equation}
where the Hamiltonian operators read 
\[
\mathcal P_0 = \begin{pmatrix}0 & \qp_x\\\qp_x & 0\end{pmatrix},\quad \mathcal P_1 =  \begin{pmatrix}-2w\qp_x-w_x & -u\qp_x+\qp_x^2\\-u\qp_x-u_x-\qp_x^2 & -\frac12\qp_x\end{pmatrix},
\]
and the Hamiltonians are defined by 
\[
H_k = \frac{1}{k+1}\int\res L^{k+1},\quad k\geq 0.
\]

Any bi-Hamiltonian integrable hierarchy admits a super extension \cite{liu2020super,liu2022variational}. In this example, let us view $(w,u)$ as local coordinates of a 2-dimensional manifold $M$, and denote by $(w,u,\qth,\qs)$ be canonical local coordinates of   $\hat M$, then the bi-Hamiltonian structure $(P_0,P_1)$ is represented as
\[
P_0 = \int \qth\qs^1,\quad P_1 = \int -w\qth\qth^1-u\qth\qs^1+\qth\qs^2-\frac14\qth\qth^1.
\]
The 1-constrained KP hierarchy, together with its bi-Hamiltonian structure, can be viewed as a super integrable hierarchy on $J^\infty(\hat M)$ consisting of mutually commuting flows 
\[
D_{P_0},\quad D_{P_1},\quad D_{X_k},\quad k\geq 1,
\] 
where  $X_k = -[P_0, H_{k+1}]\in\hm F^1$. We still use the notation $\diff{}{t_k}$ to denote the action of $D_{X_k}$ on $\hm A$ and denote by $\diff{}{\qt_i} = D_{P_i}\in\derx^1$, then we have the following super extended 1-constrained KP hierarchy:
\begin{align*}
    &\diff{w}{t_k} =\qp_x\vard{H_{k+1}}{u},\quad \diff{u}{t_k} = \qp_x\vard{H_{k+1}}{w},\\
    &\diff{\qth}{t_k} = \diff{}{\qt_0}\vard{H_{k+1}}{w},\quad \diff{\qs}{t_k} = \diff{}{\qt_0}\vard{H_{k+1}}{u},\\
    &\diff{w}{\qt_0} = \qs^1,\quad \diff{u}{\qt_0} = \qth^1,\\
 &\diff{\qs}{\qt_0} =0,\quad  \diff{\qth}{\qt_0} = 0,\\
&\diff{w}{\qt_1} = -2w\qth^1-w_x\qth-u\qs^1+\qs^2,\quad \diff{u}{\qt_1} = -u\qth^1-u_x\qth-\frac12\qs^1-\qth^2,\\
&\diff{\qth}{\qt_1} = -\qth\qth^1,\quad \diff{\qs}{\qt_1} = -\qth\qs^1.
    \end{align*}

We observe that the 1-constrained KP hierarchy admits a reduction by setting $w(x)=0$, meaning that we set all the variables $w^{(s)} = 0$ for $s\geq 0$. Furthermore, we can explicitly describe this reduction as follows.
\begin{Lem}
    We have for $k\geq 1$
    \begin{align}
        \label{BS}
    \left.\diff{w}{t_k}\right|_{w(x)=0} &= 0,\\
    \label{BT}
\left.\diff{u}{t_k}\right|_{w(x)=0} &= \sum_{s=1}^k\binom{k}{s-1}\kk{\qp_x^{k+1-s}(ug_s)-u\qp_x^{k+1-s}g_s},
    \end{align}
    where the differential polynomials $g_s$ are as defined in \eqref{BR}.
\end{Lem}
\begin{proof}
    Let us fix an integer $k\geq 1$ and denote by $X = [L^k_+,L]$, which is a pseudo-differential operator of the form  
    \[
    X = \sum_{n\geq 1} a_n\qp_x^{-n},\quad a_n\in\hm A^0.
    \]
    It follows from the Lax pair that 
    \[
    \diff{w}{t_k} = a_1,\quad \diff{u}{t_k} = \frac{1}{w}\kk{a_{2}-u\diff{w}{t_k}+\diff{w_x}{t_k}}.
    \]   
    First it is easy to see that $L|_{w(x)=0} = \qp_x$, hence we arrive at  
    \[
    a_1|_{w(x)=0} = 0,
    \]
    which proves the identity \eqref{BS}. 

    Denote by 
    \[
    \hm A_{\mathrm{poly}}^0 =\mathbb R[w,u][[w^{(s)},u^{(s)}\colon s\geq 1]]
    \]
    the subspace of $\hm A^0$ consisting of differential polynomials whose coefficients are again polynomials in $u,w$. Then it follows from the bi-Hamiltonian formalism \eqref{BU} that 
    \[
    \diff{u}{t_k}\in\hm A_{\mathrm{poly}}^0,
    \]
    therefore we have 
    \[
    \left.\diff{u}{t_k}\right|_{w(x)=0} = \left.\diff{a_2}{w}\right|_{w(x)=0}-(u-\qp_x)\left.\diff{a_1}{w}\right|_{w(x)=0}
    \]
    Then the lemma is proved by a straightforward computation using  
    \[
    \left.\diff{X}{w}\right|_{w(x)=0} = \left.\fk{\diff{L^k_+}{w}}{L}\right|_{w(x)=0}+\left.\fk{L^k_+}{\diff{L}{w}}\right|_{w(x)=0}.
    \]
\end{proof}
It is easy to see that the reduction can be lifted to the super extension of the 1-constrained KP hierarchy by also setting $\qs(x)= 0$.
\begin{Lem}
    We have 
    \[
\left.\diff{\qs}{t_k}\right|_{w(x)=\qs(x)=0} = 0.
    \]
    \end{Lem}
\begin{proof}
    It follows from the identity \eqref{BS} that 
    \[
  \left.\vard{H_{k+1}}{u}\right|_{w(x)=0}= 0,
    \]
    Then by definition, we have 
    \[
\left.\diff{\qs}{t_k}\right|_{w(x)=\qs(x)=0} = \sum_{s\geq 0}\qth^{s+1}\diff{}{u^{(s)}}\left.\vard{H_{k+1}}{u}\right|_{w(x)=0} = 0,
    \]
    which verifies the assertion.
\end{proof}

To conclude, it follows that by setting $w(x) = \qs(x) = 0$, we have a well-defined reduction of the super extended 1-constrained KP hierarchy given by 
\begin{align*}
    &\diff{u}{\bar t_k} = \qp_x W_k,\quad \diff{\qth}{\bar t_k} = \diff{W_k}{\bar \qt_0},\\
        &\diff{u}{\bar\qt_0} = \qth^1,\quad \diff{\qth}{\bar\qt_0} = 0,\\
&\diff{u}{\bar \qt_1} = -u\qth^1-u_x\qth-\qth^2,\quad \diff{\qth}{\bar \qt_1} = -\qth\qth^1,
\end{align*}
where we have 
\[
W_k = \left.\vard{H_{k+1}}{w}\right|_{w(x)=0},\quad \qp_x W_k =  \sum_{s=1}^k\binom{k}{s-1}\kk{\qp_x^{k+1-s}(ug_s)-u\qp_x^{k+1-s}g_s}.
\]
We call this a generalized bi-Hamiltonian integrable hierarchy, since after the reduction, the original bi-Hamiltonian structure of the 1-constrained KP hierarchy reduces to a generalized bi-Hamiltonian structure whose leading term is \emph{not} of $D$-type.

In what follows, let us show that the reduced integrable hierarchy 
\begin{equation}
    \label{BV}
    \diff{u}{\bar t_k} = \sum_{s=1}^k\binom{k}{s-1}\kk{\qp_x^{k+1-s}(ug_s)-u\qp_x^{k+1-s}g_s},\quad k\geq 1\end{equation}
is precisely the Burgers hierarchy
\begin{equation}
    \label{BW}
 \diff{u}{\bar t_k} = \qp_x(\qp_x+u)^{k-1}u,\quad k\geq 1.
\end{equation}
This may be proved by a straightforward computation. However, let us take a detour and prove the above identity using the theory of generalized bi-Hamiltonian structures.
\begin{Prop}
    The Burgers hierarchy admits a super extension given by 
    \begin{align*}
                &\diff{u}{s_k} = \qp_x(\qp_x+u)^{k-1}u,\quad \diff{\qth}{s_k} = \diff{}{\qt_0}(\qp_x+u)^{k-1}u,\\
        &\diff{u}{\qt_0} = \qth^1,\quad \diff{\qth}{\qt_0} = 0;\\
&\diff{u}{\qt_1} = -u\qth^1-u_x\qth-\qth^2,\quad \diff{\qth}{\qt_1} = -\qth\qth^1,\\
    \end{align*}
\end{Prop}
\begin{proof}
    It is obvious that the two odd flows commute, and it follows from the definition that  
    \[
    \fk{\diff{}{s_k}}{\diff{}{\qt_0}} = 0.
    \]
    It remains to verify that
    \begin{equation}
        \label{AAG}
    \fk{\diff{}{s_k}}{\diff{}{\qt_1}} = 0.
    \end{equation}
    First we check that 
    \[
    \diff{}{\qt_0}\comp(\qp_x+u)^k = (\qp_x+u)^k\comp\diff{}{\qt_0}+\sum_{n=1}^k (-1)^{n-1}\binom{k}{n}(\qp_x+u)^{k-n}\comp\qth^n,
    \]
    which can be proved straightforwardly using induction. Indeed, we have 
    \begin{align*}
        \diff{}{\qt_0}\comp(\qp_x+u)^{k+1}=\,& (\qp_x+u)^k\comp\diff{}{\qt_0}\comp(\qp_x+u)\\&+\sum_{n=1}^k (-1)^{n-1}\binom{k}{n}(\qp_x+u)^{k-n}\comp\qth^n\comp(\qp_x+u)\\
        =\,&(\qp_x+u)^k\comp\kk{(\qp_x+u)\diff{}{\qt_0}+\qth^1}\\
        &+\sum_{n=1}^k (-1)^{n-1}\binom{k}{n}(\qp_x+u)^{k-n}\comp\kk{(\qp_x+u)\comp\qth^n-\qth^{n+1}}\\
        =\,&(\qp_x+u)^{k+1}\comp\diff{}{\qt_0}+\sum_{n=1}^{k+1} (-1)^{n-1}\binom{k+1}{n}(\qp_x+u)^{k+1-n}\comp\qth^n.
    \end{align*}
    Using a similar proof as above, it is easy to see that 
    \begin{align*}
    \diff{}{\qt_1}\comp(\qp_x+u)^k =\,& (\qp_x+u)^k\comp\diff{}{\qt_1}+\sum_{n=1}^k (-1)^{n-1}\binom{k}{n}(\qp_x+u)^{k-n}\comp\kk{\qp_x^n(-u\qth-\qth^1)},\\
    \qp_x\comp(\qp_x+u)^k =\,& (\qp_x+u)^k\comp\qp_x+\sum_{n=1}^k (-1)^{n-1}\binom{k}{n}(\qp_x+u)^{k-n}\comp u^{(n)},\\
    \qth\comp(\qp_x+u)^k =\,&\sum_{n=0}^k (-1)^n\binom{k}{n}(\qp_x+u)^{k-n}\comp\qth^n.
    \end{align*}

    Denote by $Z_k = (\qp+u)^k1$, then it is easy to see that the desired identity \eqref{AAG} is equivalent to 
    \[
    -\diff{Z_k}{\qt_1} = \qth\qp_xZ_k+(u+\qp_x)\diff{Z_k}{\qt_0}.
    \] 
    Hence, it follows that
    \begin{align*}
        -\diff{Z_k}{\qt_1} =\,& \sum_{n=1}^k (-1)^{n-1}\binom{k}{n}(\qp_x+u)^{k-n}\qp_x^n(u\qth+\qth^1),\\
        (u+\qp_x)\diff{Z_k}{\qt_0}=\,&\sum_{n=1}^k (-1)^{n-1}\binom{k}{n}(\qp_x+u)^{1+k-n}\qth^n\\
        =\,&\sum_{n=1}^k (-1)^{n-1}\binom{k}{n}(\qp_x+u)^{k-n}(\qth^{n+1}+u\qth^n),\\
        \qp_x Z_k=\,&\sum_{n=1}^k (-1)^{n-1}\binom{k}{n}(\qp_x+u)^{k-n} u^{(n)}.
    \end{align*}
    We then need to verify the identity
    \[
    \qth\qp_xZ_k = \sum_{n=1}^k (-1)^{n-1}\binom{k}{n}(\qp_x+u)^{k-n} \kk{\qp_x^n(u\qth)-u\qth^n}.
    \] 
    By a direct computation, we arrive at 
    \begin{align*}
\qth\qp_xZ_k =\,&\qth\sum_{n=1}^k (-1)^{n-1}\binom{k}{n}(\qp_x+u)^{k-n}u^{(n)}\\
=\,&\sum_{n=1}^k\sum_{j=0}^{k-n} (-1)^{j+n-1}\binom{k}{n}\binom{k-n}{j}(\qp_x+u)^{k-n-j}\kk{\qth^j u^{(n)}}\\
=\,&\sum_{p=1}^k\sum_{q=1}^p(-1)^{p-1}\binom{k}{q}\binom{k-q}{p-q}(\qp_x+u)^{k-p}\kk{\qth^{p-q} u^{(q)}}\\
=\,&\sum_{p=1}^k\sum_{q=1}^p(-1)^{p-1}\binom{k}{p}\binom{p}{q}(\qp_x+u)^{k-p}\kk{\qth^{p-q} u^{(q)}}\\
=\,&\sum_{p=1}^k(-1)^{p-1}\binom{k}{p}(\qp_x+u)^{k-p}\kk{\qp_x^{p}(u\qth)-u\qth^p},
    \end{align*}
here for the second equality, we perform a re-summation using the indices $p:=n+j$ and $q:=n$. Therefore, we show the validity of \eqref{AAG}. 

Finally, it follows from Theorem \ref{AH} that the even flows mutually commute with each other, and this concludes the proof of the proposition.
\end{proof}
\begin{Th}
    The reduced hierarchy of the 1-constrained KP hierarchy by setting $w(x) = 0$ gives the Burgers hierarchy.
    \end{Th}
    \begin{proof}
        We have seen that both the flows \eqref{BV} and \eqref{BW}
        commute with the generalized bi-Hamiltonian structure 
        \begin{align*}
                    &\diff{u}{\qt_0} = \qth^1,\quad \diff{\qth}{\qt_0} = 0,\\
&\diff{u}{\qt_1} = -u\qth^1-u_x\qth-\qth^2,\quad \diff{\qth}{\qt_1} = -\qth\qth^1.
        \end{align*}
        It follows from Theorem \ref{AH} that the flows \eqref{BV} and \eqref{BW} coincide since they share a same hydrodynamic leading term. 
    \end{proof}

    It's worth noting that the Burgers hierarchy is an integrable hierarchy of double ramification type, and it may control some version of F-cohomological field theory on the space of Riemann surfaces with boundaries \cite{arsie2021flat}.
    \begin{Rem}
        The formulation of Burgers hierarchy as a generalized bi-Hamiltonian structure is compatible with the recursion structure. Define the differential operator 
        \[
        \mathcal D_0 = \qp_x,\quad \mathcal D_1 = u\qp_x+u_x+\qp_x^2
        \] 
        which is characterized by 
        \[
        \diff{u}{\qt_0} = \mathcal D_0(\qth),\quad \diff{u}{\qt_1} = -\mathcal D_1(\qth).
        \]
        These operators do not form a bi-Hamiltonian structure, but it still makes sense to consider the corresponding recursion operator $\mathcal R$, which is a formal pseudo-differential operator defined by 
        \[
        \mathcal R = \mathcal D_1\comp \mathcal D_0^{-1} = \qp_x+u+u_x\qp_x^{-1}.
        \]
        Then it follows that 
        \[
        \mathcal R \diff{u}{t_k} = (u\qp_x+u_x+\qp_x^2)(\qp_x+u)^{k-1}u = \qp_x(\qp_x+u)^k u = \diff{u}{t_{k+1}}.
        \]
        This indicates that it might be possible to extend the construction of super tau-cover \cite{liu2020super,liu2022variational} to the case of generalized bi-Hamiltonian systems.

    \end{Rem}

\subsection{Integrable conservation laws}

Scalar generalized bi-Hamiltonian structurew can be used to study integrable conservation laws. Recall that an integrable conservation law is an evolutionary PDE of the form 
\[
\diff{u}{t} = \qp_x\big(u^2+\sum_{k\geq 1}A_k\big),\quad A_k\in\hm A^0_k,
\]
such that it is formally integrable, meaning that for any function $f(u)$, there exists a symmetry of the above PDE of the form 
\[
\diff{u}{t_f} = \qp_x\big(f(u)+\sum_{k\geq 1}B_k\big),\quad B_k\in\hm A^0_k.
\]
When the term $A_1\neq 0$, the corresponding conservation law is called viscous \cite{arsie2015integrable}. In \cite{arsie2015integrable}, the classification of integrable viscous conservation laws are studied, and they conjectured that the Miura-equivalence class of such systems are parametrized exactly by the term $A_1$. If we write $A_1 = 2a(u)u_x$, then the function $a(u)$ is called the viscous central invariant of a viscous integrable conservation law. Using the theory of generalized bi-Hamiltonian structure, we can partially verify this conjecture.
\begin{Th}[=Theorem \ref{AK}]
    There exists an integrable viscous conservation law for any non-zero viscous central invariant $a(u)$.
\end{Th}
\begin{proof}
    Let us consider the following generalized bi-Hamiltonian structure of viscous type:
\begin{align*}
\diff{u}{\qt_0^{[0]}} &= \qth^1,\quad \diff{\qth}{\qt_0^{[0]}} = 0,\\
\diff{u}{\qt_1^{[0]}} &= u\qth^1+u_x\qth,\quad \diff{\qth}{\qt_1^{[0]}} = \qth\qth^1.
\end{align*}
It follows from Theorem \ref{AG} and Theorem \ref{BP} that for any function $a(u)$, we can construct a deformation of $(\diff{}{\qt_0^{[0]}},\diff{}{\qt_1^{[0]}})$ given by 
\begin{align*}
\diff{u}{\qt_0} &= \qth^1,\quad \diff{\qth}{\qt_0} = 0,\\
\diff{u}{\qt_1} &= u\qth^1+u_x\qth+2a(u)\qth^2+2a'(u)u_x\qth^1+\hm A^1_{\geq 3},\quad \diff{\qth}{\qt_1} = \qth\qth^1+\hm A^2_{\geq 3}.
\end{align*}
Note that we have assumed that $\diff{}{\qt_0^{[0]}}=\diff{}{\qt_0}$ by using Corollary \ref{BX}. Using the fact that cohomology groups 
\[
BH^1_{\geq 3}\kk{\derx,\diff{}{\qt_0^{[0]}},\diff{}{\qt_1^{[0]}}} = 0,
\]
it is not hard to prove that for any smooth function $f(u)$, there exists a flow $\diff{}{t_f}\in\derx^0$ such that 
\[
\fk{\diff{}{\qt_0}}{\diff{}{t_f}} = \fk{\diff{}{\qt_1}}{\diff{}{t_f}} = 0
\]
with the form 
\[
\diff{u}{t_f} = \qp_x f(u)+\hm A^0_{\geq 2}.
\]
Furthermore, we have
\[
\fk{\diff{}{t_f}}{\diff{}{t_g}} = 0,\quad \forall\ f,g\in C^\infty(u).
\]

It is easy to observe that each flow $\diff{}{t_f}$ is actually of the form 
\[\diff{u}{t_f} = \qp_x\big(f(u)+\sum_{k\geq 1}B_k\big),\quad B_k\in\hm A^0_k,\]
and an explicit calculation implies that 
\[
\diff{u}{t_{u^2}} =  \qp_x\left(u^2+2a(u)u_x+\hm A^0_{\geq 2}\right).
\]
The theorem is proved.
\end{proof}

\subsection{An example of exceptional generalized bi-Hamiltonian structure}
\label{AE}
In Sect.\,\ref{BY}, we have studied in detail the deformation of scalar semisimple generalized bi-Hamiltonian structure except for the exceptional case. In this section, we show, by an explicit computation, that deformations of generalized bi-Hamiltonian structures of exceptional type can be complicated.

Let us consider the following generalized bi-Hamiltonian structure:
\begin{align*}
\diff{u}{\qt_0} &= \qth^1,\quad \diff{\qth}{\qt_0} = 0,\\
\diff{u}{\qt_1} &= u\qth^1,\quad \diff{\qth}{\qt_1} = 0,
\end{align*}
with the function $\bar K(u) = 0$. In what follows, we will compute the space of infinitesimal deformations of the above generalized bi-Hamiltonian structures.
\begin{Th}
    \label{CN}
    \[
    H^1_d\left(\derx,\diff{}{\qt_0},\diff{}{\qt_1}\right) \cong \begin{cases}
			0, & \text{if $d = 2,3,4$ or $d = 2n$ for $n\geq 3$,}\\
            C^\infty(u), & \text{otherwise.}
		 \end{cases}
    \]
\end{Th}

We use the same framework as that presented in Sect.\,\ref{BY}. Firstly, we transform the problem into the computation of cohomology groups 
\[
H^p_d(\hm B,D),\quad  \hm B= \bigoplus_{p,d}\hm B^p_d[\ql],\quad \hm B^p_d = \hm A^p_d[\ql]\oplus\hm A^{p+1}_d[\ql],
\]
 with the differential given by 
    \[
        D = \begin{pmatrix}D_1 & D_2\\0 & D_4\end{pmatrix},\quad D_1=\diff{}{\qt(\ql)}-\qth^1,\quad D_2=(-1)^{p+1}(u-\ql)\qp_x,\quad  D_4=\diff{}{\qt(\ql)}.
    \]
In particular, it follows that the complex $(\hm B, D)$ is the mapping cone complex of the cochain map 
\[
\rho^p \colon \kk{\hm A^p[\ql],D_4}\to \kk{\hm A^p[\ql],D_1},\quad  \rho^p = (-1)^p(u-\ql)\qp_x.
\]
Hence, there exists short exact sequence 
        \xym{
    0\ar[r] & \cok\rho^1_*\ar[r]^{i_*} & H^1(\hm B, D)\ar[r]^{\pi_*} & \ker\rho^2_*\ar[r]& 0,
        }
    where $\rho_*$ is the induced map on cohomology groups and $i_*$ is the inclusion to the first component and $\pi_*$ is the projection to the second component.

Let us first compute $H^p_d(\hm A[\ql],D_4)$. Using Lemma \ref{BZ}, it suffices to consider 
\[
BH^p_d\left(\hm A,\diff{}{\qt_0},\diff{}{\qt_1}\right) = \frac{\hm A^p_d\cap\ker\diff{}{\qt_0}\cap \ker\diff{}{\qt_1}}{\hm A^p_d\cap\ima \diff{}{\qt_1}\comp \diff{}{\qt_0}},
\]
which is slightly easier to work with in this particular example. Note that in this case, we have 
\[
\diff{}{\qt_0} = \sum_{s\geq 0}\qth^{s+1}\diff{}{u^{(s)}},\quad \diff{}{\qt_1} = \sum_{s\geq 0}\qp_x^{s}(u\qth^1)\diff{}{u^{(s)}}.
\]
It follows that we can decompose the space 
\[
\hm A = \hm G \oplus\qth\hm G,\quad \hm G=\big\{f\in\hm A\colon \diff{f}{\qth} = 0\big\},
\]
and it is obvious that 
\begin{equation}
    \label{CA}
BH^p_d\left(\hm A,\diff{}{\qt_0},\diff{}{\qt_1}\right) = BH^p_d\left(\hm G,\diff{}{\qt_0},\diff{}{\qt_1}\right)\oplus \qth BH^{p-1}_d\left(\hm G,\diff{}{\qt_0},\diff{}{\qt_1}\right).
\end{equation}
For this reason and to sanitize the notations a bit, in what follows we introduce 
\[
\xi_i = \qth^{i+1},\quad i\geq 0.
\]
Then we have 
\[
\diff{}{\qt_0} = \sum_{i\geq 0}\xi_i\diff{}{u^{(i)}},\quad \diff{}{\qt_1} = \sum_{i\geq 0}\xi_i\mathcal D_i,\quad \diff{}{\qt_1}\comp \diff{}{\qt_0} = \sum_{0\leq i<j}\xi_i\xi_j\mathcal P_{ij},
\]
where we denote
\[
\mathcal D_i = \sum_{s\geq i}\binom{s}{i}u^{(s-i)}\diff{}{u^{(s)}},\quad \mathcal P_{ij} = \mathcal D_i\comp\diff{}{u^{(j)}}-\mathcal D_j\comp\diff{}{u^{(i)}}.
\]
Let us further define the spaces 
\[
\hm G^{(N)} = \big\{f\in\hm G\colon \diff{f}{u^{(n)}} = \diff{f}{\xi_n} = 0,\quad \forall\ n\geq N+1\big\}.
\]
\begin{Lem}
    \label{CI}
For $d\geq 2$, we have 
\[
BH^0_d\left(\hm G,\diff{}{\qt_0},\diff{}{\qt_1}\right) = 0,\quad BH^1_d\left(\hm G,\diff{}{\qt_0},\diff{}{\qt_1}\right) \cong C^\infty(u),
\]
where the second isomorphism reads
\[
C^\infty(u)\to BH^1_d\left(\hm G,\diff{}{\qt_0},\diff{}{\qt_1}\right),\quad f(u)\mapsto \qp_x^{d-1}(f(u)\xi_0).
\]
\end{Lem}
\begin{proof}
For any $f\in\hm A^0_d$, with $d\geq 2$, if we have 
\[
\diff{f}{\qt_0} = \sum_{i\geq 0}\xi_i\diff{f}{u^{(i)}} = 0,
\]
then it is obvious that $f = 0$, and this proves the first assertion.

For the second assertion, take a cocycle $\qo\in\hm G^1_d$ with 
\[
\diff{\qo}{\qt_0} = \diff{\qo}{\qt_1} = 0.
\]
It follows from the triviality of Hamiltonian cohomology \cite{getzler2002darboux} that there exists some $f\in\hm G^0_{d-1}$ such that 
\[
\qo = \diff{f}{\qt_0},\quad \diff{}{\qt_1}\diff{f}{\qt_0} = 0,
\]
which implies that 
\[
\mathcal P_{ij}(f) = 0,\quad \forall\ 0\leq i< j.
\]
Assume $f\in\hm G^{(N)}$ for some $N>0$, then it follows that 
\[
0 = \mathcal P_{0N}(f) = \mathcal D_0\diff{f}{u^{(N)}}-\mathcal D_N\diff{f}{u} = \sum_{i\geq 1}u^{(i)}\diff{}{u^{(i)}}\kk{\diff{f}{u^{(N)}}} = 0.
\]
This shows that $\diff{f}{u^{(N)}}$ is actually a function in $u$, and in particular it must vanish if $N\neq d-1$. So we conclude that $f = 0$ if 
\[
\diff{f}{u^{(d-1)}} = 0.
\]

Now let us consider the class 
\[
\tilde\qo = \diff{\tilde f}{\qt_0},\quad \tilde f = f-\qp_x^{d-1}g(u),\quad g(u) = \int \diff{f}{u^{(d-1)}} du.
\]
It is easy to see that $\qo'$ is indeed a cocycle, and it follows from the construction that
\[
\diff{\tilde f}{u^{(d-1)}} = 0,
\]  
therefore we conclude that $\tilde f = 0$ and hence the lemma holds true.
\end{proof}
By applying a similar idea used in the above proof, we can compute all other cohomology groups. In what follows, let us denote by $\qa_{i_1,\dots,i_p}$ the coefficient of $\xi_{i_1}\dots\xi_{i_p}$ for any $\qa\in\hm G^p$. Then for any $\qo\in\hm G^p$ and $i_0<i_1<\dots<i_p$, we have 
\begin{align}
\label{CB}
\kk{\diff{\qo}{\qt_0}}_{i_0,\dots,i_p} &= \sum_{j=0}^p (-1)^j\diff{}{u^{(i_j)}}\qo_{i_0,\dots,\hat i_j,\dots,i_p},\\
\label{CC}
\kk{\diff{\qo}{\qt_1}}_{i_0,\dots,i_p} &= \sum_{j=0}^p (-1)^j\mathcal D_j \qo_{i_0,\dots,\hat i_j,\dots,i_p}.
\end{align} 

\begin{Lem}
    \label{CO}
    For $p\geq 2$ and $d\geq 2$ we have an isomorphism
    \[
    \bigoplus_{i=1}^{n(p,d)} C^\infty(u) \xrightarrow{\cong} BH^p_d\left(\hm G,\diff{}{\qt_0},\diff{}{\qt_1}\right),
    \]
    where $n(p,d)$ is the number of integer solutions of the equation 
    \[
    i_1+i_2+\dots+i_p = d-p,\quad 0\leq i_1<i_2<\dots<i_p.
    \] 
    For given $n(p,d)$ functions $(f_{i_1,\dots, i_p})$ where $i_1,\dots,i_p$ satisfy the above equation, the isomorphism reads 
    \[
    (f_{i_1,\dots, i_p})\mapsto \sum_{i_1,\dots,i_p}\qp_x^{i_1}(f_{i_1,\dots,i_p}\xi_0)\xi_{i_2}\dots\xi_{i_p}.
    \]     
\end{Lem}
\begin{proof}
    Define a morphism of $C^\infty(u)$-algebras 
    \[
    \pi\colon \hm G\to C^\infty(u)\otimes\bigwedge[\xi_0,\xi_1,\dots],\quad u^{(s+1)}\mapsto 0,\quad \xi_s\mapsto \xi_s,\quad s\geq 0.
    \]
    It is easy to see, by applying the definition of $\mathcal P_{ij}$, that the map $\pi$ annihilates coboundaries, and hence $\pi$ is well-defined on cohomologies. It suffices to show that the induced map 
    \[
    \pi\colon BH^p\left(\hm G,\diff{}{\qt_0},\diff{}{\qt_1}\right)\to C^\infty(u)\otimes\bigwedge\nolimits^p[\xi_0,\xi_1,\dots]
    \]
    is an isomorphism.

    The above map is surjective, due to the identity 
    \[
    \pi\kk{\qp_x^{i_1}(f(u)\xi_0)\xi_{i_2}\dots\xi_{i_p}} = f(u)\xi_{i_1}\xi_{i_2}\dots\xi_{i_p},\quad \forall\ f(u)\in C^\infty(u).
    \]
    It only remains to verify that $\pi$ is injective. Take a cocycle $\qo\in\hm G^p$ with
    \[\diff{\qo}{\qt_0} = \diff{\qo}{\qt_1} = 0,\quad \pi(\qo) = 0.\]
    Assume that $\qo\in\hm G^{(N)}$ for some $N>0$, and for any $0<i_2<\dots<i_{p-1}<N$, let us denote by 
    \[
    g = \qo_{0,i_2,\dots,i_{p-1}, N},\quad g\in\hm A^0.
    \]
   If $g\neq 0$, due to the assumption that $\pi(g) = 0$, it follows that there exists $h\in\hm A^0$ such that 
    \[
    \sum_{i\geq 1}u^{(i)}\diff{}{u^{(i)}}\diff{h}{u^{(N)}} = (-1)^p g,\quad h\in\hm G^{(N)}.
    \]
    Let us define the coboundary
    \[
    \qb = \diff{}{\qt_1}\diff{}{\qt_0}\kk{h\xi_{i_2}\dots\xi_{i_{p-1}}},
    \]
    then it follows that 
   \[ 
    \qb_{0,i_2,\dots,i_{p-1}, N} = (-1)^p \mathcal P_{0,N}(h) = g,\quad \qb\in\hm G^{(N)}.
    \]
    Therefore, after replacing the cocycle $\qo$ by $\qo-\qb$, we may assume that
    \[
    \qo_{0,i_2,\dots,i_{p-1}, N} = 0,\quad 0<i_2<\dots<i_{p-1}<N.
    \] 
    Now for any $0<i_1<\dots<i_{p-1}<N$, it follows from \eqref{CB} and \eqref{CC} that 
    \begin{align*}
    \kk{\diff{\qo}{\qt_0}}_{0,i_1,\dots,i_{p-1},N} &= \diff{\qo_{i_1,\dots,i_{p-1},N}}{u}+(-1)^p \diff{\qo_{0,i_1,\dots,i_{p-1}}}{u^{(N)}},\\
    \kk{\diff{\qo}{\qt_1}}_{0,i_1,\dots,i_{p-1},N} &= \mathcal D_0\qo_{i_1,\dots,i_{p-1},N}+(-1)^p \mathcal D_N\qo_{0,i_1,\dots,i_{p-1}}\\
    &=\sum_{i\geq 0}u^{(i)}\diff{\qo_{i_1,\dots,i_{p-1},N}}{u^{(i)}}+(-1)^p u\diff{\qo_{0,i_1,\dots,i_{p-1}}}{u^{(N)}}.
    \end{align*}
    We conclude that 
    \[
    \sum_{i\geq 1}u^{(i)}\diff{\qo_{i_1,\dots,i_{p-1},N}}{u^{(i)}} = 0,
    \]
    which implies that $\qo_{i_1,\dots,i_{p-1},N}=0$ since we assume that $\pi(\qo) = 0$. Hence, we have  
    \[
    \diff{\qo}{\xi_N} = 0,
    \]
    but we also have 
    \[
    \kk{\diff{\qo}{\qt_0}}_{i_1,\dots,i_{p},N} = (-1)^{p}\diff{\qo_{i_1,\dots,i_{p}}}{u^{(N)}} = 0,\quad 0\leq i_1<\dots<i_p<N.
    \]
    It follows that $\qo\in\hm G^{(N-1)}$, and by induction we conclude that $\qo = 0$, hence $\pi$ is injective, and the lemma holds true.
\end{proof}

Next let us proceed to computing the cohomology groups $H^p_d(\hm A[\ql],D_1)$. By applying Lemma \ref{BZ}, we again work with the cohomology groups  
\[
BH^p_d\left(\hm A,\diff{}{\qt_0},\diff{}{\qt_1}-\qth^1\right)
\]
The decomposition 
\[
\hm A = \hm G\oplus\qth\hm G
\]
is also a decomposition subcomplex, and hence 
\begin{equation}
    \label{CD}
BH^p_d\left(\hm A,\diff{}{\qt_0},\diff{}{\qt_1}-\qth^1\right) = BH^p_d\left(\hm G,\diff{}{\qt_0},\diff{}{\qt_1}-\qth^1\right)\oplus \qth BH^{p-1}_d\left(\hm G,\diff{}{\qt_0},\diff{}{\qt_1}-\qth^1\right).
\end{equation}
For this reason, we will continue using the notations $\xi_i = \qth^{i+1}$.

We start by proving a technical lemma that will be used in the computation of cohomology groups.
\begin{Lem}
    \label{CE}
For any smooth function $g(u)$ and $d\geq 2$, there exists a differential polynomial $\Psi_d[g]\in\hm G^0_d$, such that $\Psi_2[g] = \frac{1}{2}g(u)u_x^2$, and 
\[
\Psi_{d+1}[g] = g(u)u_xu^{(d)}+\hm G^{(d-1)},\quad \kk{\diff{}{\qt_1}-\xi_0}\diff{\Psi_d[g]}{\qt_0} = 0,\quad d\geq 2.
\]
\end{Lem}
\begin{proof}
Define the differential polynomial by 
\begin{equation}
    \label{CH}
\Psi_d[g] = \frac{1}{d}\sum_{r=0}^{d-2}(r+1)\frac{d^r g}{d u^r}B_{d,r+2}(u_x,u_{xx},\dots, u^{d-1-r}),
\end{equation}
where $B_{n,k}$ are the (partial) exponential Bell polynomials with the generating function 
\begin{equation}
    \label{CG}
\sum_{n\geq 0}\frac{z^n}{n!}\sum_{k=0}^n \qe^kB_{n,k} = \exp\kk{\qe\sum_{s\geq 1}\frac{z^s}{s!}u^{(s)}}.
\end{equation}
It is straightforward to verify that indeed 
\[
\Psi_2[g] = \frac{1}{2}g(u)u_x^2,\quad \Psi_d[g] = g(u)u_xu^{(d-1)}+\hm G^{(d-2)},\quad d\geq 3.
\]
Let us define the generating function
\[
\Phi(z) = \sum_{d\geq 2}d\Psi_d[g]\frac{z^d}{d!} =\sum_{r\geq 0} \frac{d^r g}{d u^r}\frac{u^\bullet(z)^{r+2}}{(r+2)r!},\quad u^\bullet(z) = \sum_{s\geq 1}\frac{z^s}{s!}u^{(s)},
\]
where we have used the generating function \eqref{CG}. Note that we have 
\begin{align*}
    \diff{}{u}u^\bullet(z) &= 0,\quad \mathcal D_0 u^\bullet(z) = u^\bullet(z)\\
\diff{}{u^{(k)}}u^\bullet(z) &= \frac{z^k}{k!},\quad \mathcal D_k u^\bullet(z) =\frac{z^k}{k!} \kk{u^\bullet(z)+u},\quad k\geq 1.
\end{align*}
Therefore, for $1\leq i<j$, 
\begin{align*}
    \mathcal P_{ij}\Phi(z) =&\, \sum_{r\geq 0}\frac{d^r g}{d u^r}\mathcal P_{ij}\kk{\frac{u^\bullet(z)^{r+2}}{(r+2)r!}}\\
     =&\, \sum_{r\geq 0}\frac{d^r g}{d u^r}\kk{\frac{z^j}{j!}\mathcal D_i\frac{u^\bullet(z)^{r+1}}{r!}-\frac{z^i}{i!}\mathcal D_j\frac{u^\bullet(z)^{r+1}}{r!}}\\
     =&\,0,\\
     \mathcal P_{0j}\Phi(z) =&\,\sum_{r\geq 0}\mathcal D_0\kk{\frac{z^j}{j!}\frac{d^r g}{d u^r}\frac{u^\bullet(z)^{r+1}}{r!}}-\mathcal{D}_j\kk{\frac{d^{r+1} g}{d u^{r+1}}\frac{u^\bullet(z)^{r+2}}{(r+2)r!}}\\
     =&\,\sum_{r\geq 0}\frac{z^j}{j!}u\frac{d^{r+1} g}{d u^{r+1}}\frac{u^\bullet(z)^{r+1}}{r!}+\frac{z^j}{j!}\frac{d^r g}{d u^r}\frac{(r+1)u^\bullet(z)^{r+1}}{r!}\\
     &-\sum_{r\geq 0}\frac{z^j}{j!}\frac{d^{r+1} g}{d u^{r+1}}\frac{u^\bullet(z)^{r+2}+uu^\bullet(z)^{r+1}}{r!}\\
     =&\,\sum_{r\geq 0}\frac{z^j}{j!}\frac{d^{r} g}{d^r u}\frac{u^\bullet(z)^{r+1}}{r!}\\
     =&\,\diff{\Phi[z]}{u^{(j)}}.
\end{align*}
It follows that 
\[
\diff{}{\qt_1}\diff{\Phi[z]}{\qt_0} = \xi_0 \diff{\Phi[z]}{\qt_0},
\]
which shows that $\Psi_d[g]$ defined in \eqref{CH} satisfy the requirement. The lemma is proved.
\end{proof}
\begin{Lem}
    \label{CM}
For $d\geq 2$, we have 
\[
BH^0_d\left(\hm G,\diff{}{\qt_0},\diff{}{\qt_1}-\xi_0\right) = 0,\quad BH^1_{d+1}\left(\hm G,\diff{}{\qt_0},\diff{}{\qt_1}-\xi_0\right) \cong C^\infty(u),
\]
where the second isomorphism reads
\[
C^\infty(u)\to BH^1_{d+1}\left(\hm G,\diff{}{\qt_0},\diff{}{\qt_1}\right),\quad f(u)\mapsto \diff{\Psi_{d}[f]}{\qt_0}.
\]
Furthermore, 
\[
BH^1_2\left(\hm G,\diff{}{\qt_0},\diff{}{\qt_1}-\xi_0\right) = 0.
\]
\end{Lem}
\begin{proof}
    The overall strategy is similar to that of the proof of Lemma \ref{CI} and in particular the first assertion is proved in the same way. Let us prove the second assertion. Take a cocycle $\qo\in \hm G^1_{d+1}$ with 
    \[
    \diff{\qo}{\qt_0} = \diff{\qo}{\qt_1}-\xi_0\qo = 0.
    \]
    Using the triviality of Hamiltonian cohomology,  we can write $\qo = \diff{f}{\qt_0}$ for some $f\in \hm G^0_{d}$ satisfying
    \[
    A:=\kk{\diff{}{\qt_1}-\xi_0}\diff{f}{\qt_0} = 0.
    \]
    Assume that $f\in\hm G^{(N)}$ for some $N\geq 1$, then 
    \[
    A_{0,N} = \mathcal P_{0N}(f)-\diff{f}{u^{(N)}} = \sum_{j\geq 1}u^{(j)}\diff{}{u^{(j)}}\kk{\diff{f}{u^{(N)}}} - \diff{f}{u^{(N)}} = 0,
    \]
    from which it follows that $\diff{f}{u^{(N)}}$ is of length 1, and hence we conclude that 
    \[
    f = g(u)u^{(N)}u^{(d-N)}+\hm G^{(N-1)}
    \]
    for some $g(u)$. Note that in particular we must have $0<d-N\leq N$. Therefore, if $d+1 = 2$, it is necessary that $f = 0$, hence we have 
    \[BH^1_2\left(\hm G,\diff{}{\qt_0},\diff{}{\qt_1}-\xi_0\right) = 0.\]
    
    In what follows we assume $d+1\geq 3$.
    For $1\leq i< N$, we arrive at 
    \begin{align*}
    A_{i,N} = \mathcal P_{iN}(f) &= \sum_{s\geq i}\binom{s}{i}u^{(s-i)}\diff{}{u^{(s)}}\diff{f}{u^{(N)}}-u\diff{}{u^{(N)}}\diff{f}{u^{(i)}}\\
     &= \binom{d-N}{i}u^{(d-N-i)}g(u)-ug(u)\qd_{i,d-N}.
    \end{align*}
    If $d-N\neq 1$, we must have $N\geq 2$, so we can take $i = 1$ in the identity above to obtain $g(u) = 0$. Therefore, we conclude that $g(u)$ is possibly nonzero only if $N = d-1$. In this case, consider 
    \[
\tilde f = 
\begin{cases}
    f-\Psi_2[2g] & d=2\\
    f-\Psi_{d}[g] & d\geq 3.
\end{cases}
\]
Then we see that $\tilde \qo = \diff{\tilde f}{\qt_0}$ is again a cocycle due to Lemma \ref{CE} and $\tilde f\in\hm G^{(d-2)}$. Then the argument above applied to $\tilde \qo$ implies that $\tilde \qo = 0$. The lemma is proved.
\end{proof}
\begin{Lem}
    \label{CJ}
    Let $f\in\hm G$ be a monomial of length $\ell$. Then all monomials in the coboundary 
    \[
    B:=\kk{\diff{}{\qt_1}-\xi_0}\diff{f}{\qt_0}
    \]
    have length at least $\ell-1$.
\end{Lem}
\begin{proof}
    First, it is easy to see that all monomials in $B$ have length at least $\ell-2$. However, the monomials that have length $\ell-2$ are given by 
    \[
    \sum_{i\geq 1}u\xi_i\diff{}{u^{(i)}}\kk{\sum_{j\geq 1}\xi_j\diff{f}{u^{(j)}}} = \sum_{i,j\geq 1}u\xi_i\xi_j\frac{\qp^2 f}{\qp u^{(i)}\qp u^{(j)}}=0.
    \]  
    The lemma is proved.
\end{proof}
\begin{Lem}
    \label{CK}
    Let $\qo\in\hm G^2_d$ for $d\geq 2$, and assume that $\qo$ contains neither length = 0 monomials nor length = 1 monomials. If 
    \[
    \diff{\qo}{\qt_0} = \diff{\qo}{\qt_1}-\xi_0\qo = 0,
    \] 
    then $[\qo]$ defines a trivial cohomology class in 
    \[BH^2_d\left(\hm G,\diff{}{\qt_0},\diff{}{\qt_1}-\xi_0\right).\]
\end{Lem}
\begin{proof}
    Assume that $\qo\in\hm G^{(N)}$ for some $N\geq 1$. If we take arbitrary 
    \[f\in\hm G^0\cap\hm G^{(N)},\]
    it follows that the coboundary 
    \[
B =\kk{\diff{}{\qt_1}-\xi_0}\diff{f}{\qt_0}
    \]
    generated by $f$ satisfies
    \[
    B_{0,N} = \sum_{j\geq 1}u^{(j)}\diff{}{u^{(j)}}\kk{\diff{f}{u^{(N)}}} - \diff{f}{u^{(N)}}.
    \]
    Since all monomials of $\qo_{0,N}$ have length at least two, it follows that we can choose suitable $f$ whose monomials have length at least 3, such that 
    \[
    \qo_{0,N} = B_{0,N}.
    \]
    Note that all monomials in $B$ have length at least 2 due to Lemma \ref{CJ}, hence by replacing  $\qo$ with $\qo-B$, we can assume that $\qo_{0,N} = 0$.

    Now fix $0<i<N$, it follows that 
    \begin{align*}
        \kk{\diff{\qo}{\qt_0}}_{0,i,N} &= \diff{\qo_{i,N}}{u}+\diff{\qo_{0,i}}{u^{(N)}} = 0,\\
        \kk{\diff{\qo}{\qt_1}-\xi_0\qo}_{0,i,N} &= (\mathcal D_0-1)\qo_{i,N}+\mathcal D_N \qo_{0,i} = 0.
    \end{align*}
    By our assumption that $\qo\in\hm G^{(N)}$, we know that 
    \[
    \mathcal D_N\qo_{0,i} = u\diff{\qo_{0,i}}{u^{(N)}},
    \]
    hence it follows that 
    \[
    \sum_{j\geq 1}u^{(j)}\diff{\qo_{i,N}}{u^{(j)}}- \qo_{i,N} = 0.
    \]
    This shows that all monomials in $\qo_{i,N}$ have length 1, so we conclude that $\qo_{i,N} = 0$ for all $0\leq i<N$.

    Finally, fix any $i<j<N$, we have 
    \[
    \kk{\diff{\qo}{\qt_0}}_{i,j,N} = \diff{\qo_{i,j}}{u^{(N)}} = 0,
    \]
    hence actually $\qo\in\hm G^{(N-1)}$. Then by induction, we see $[\qo] = 0$ as a cohomology class. The lemma is proved.
\end{proof}
\begin{Lem}
    \label{CP}
    The cohomology group 
    \[
    BH^2\left(\hm G,\diff{}{\qt_0},\diff{}{\qt_1}-\xi_0\right)
    \]
    is the linear space spanned by cocycles
    \[
    \xi_i\diff{\Psi_j[g]}{\qt_0},\quad 0\leq i\leq j,\quad i\neq 1,\quad g\in C^\infty(u),
    \]
    where $\Psi_j[g]$ are as defined in \eqref{CH}.
\end{Lem}
\begin{proof}
    Fix a cocycle $\qo\in\hm G^2$. Assume that $\qo\in\hm G^{(N)}$, then similar to the proof of Lemma \ref{CK}, we may require that all the monomials contained in $\qo_{0,N}$ has length 1. Then for $0<i<N$, we have 
    \begin{align*}
        \kk{\diff{\qo}{\qt_0}}_{0,i,N} &= \diff{\qo_{i,N}}{u}-\diff{\qo_{0,N}}{u^{(i)}}+\diff{\qo_{0,i}}{u^{(N)}} = 0,\\
        \kk{\diff{\qo}{\qt_1}-\xi_0\qo}_{0,i,N} &= (\mathcal D_0-1)\qo_{i,N}-\mathcal{D}_i  \qo_{0,N}+\mathcal D_N \qo_{0,i} = 0.
    \end{align*}
    Let us denote by 
    \[
    \bar{\mathcal D_i} = \mathcal D_i-u\diff{}{u^{(i)}} = \sum_{n\geq i+1}\binom{n}{i}u^{(n)}\diff{}{u^{n+i}},\quad i\geq 1
    \]
    then it follows that 
    \begin{equation}
        \label{CL}
    \sum_{i\geq 1}u^{(i)}\diff{\qo_{i,N}}{u^{(i)}}-\qo_{i,N} = \bar{\mathcal D_i}\qo_{0,N}.
    \end{equation}
    Since all monomials of $\bar{\mathcal D_i}\qo_{0,N}$ have length 1, we conclude that monomials in $\qo_{i,N}$ also have length 1. But this implies that 
    \[
    \bar{\mathcal D_i}\qo_{0,N} = 0,\quad 0<i<N.
    \]
    If we write 
    \[
    \qo_{0,N} = \sum_{i\geq 1}a_i(u)u^{(i)},
    \]
    then the identity above implies that 
    \[
    \sum_{k\geq 1}\binom{k+i}{i}u^{(k)}a_{i+k}(u) = 0,\quad 0
    <i<N,
    \]
    from which we arrive at 
    \[
    \qo_{0,N} = a_1(u)u_x.
    \]

    Let us further analyze how adding coboundaries changes $\qo_{i,N}$. For $f\in\hm G^0\cap\hm G^{(N)}$, let 
        \[
B =\kk{\diff{}{\qt_1}-\xi_0}\diff{f}{\qt_0}
    \]
    be the coboundary it generates. We want to preserve the length of $\qo_{0,N}$, then it is necessary that 
    \[
    \sum_{j\geq 1}u^{(j)}\diff{}{u^{(j)}}\kk{\diff{f}{u^{(N)}}} - \diff{f}{u^{(N)}} = 0.
    \]
    For such a differential polynomial $f$, it is easy to compute that 
    \[
    B_{i,N} = \bar{\mathcal D}_i \diff{f}{u^{(N)}},\quad i\geq 1.
    \]
    Note that all monomials in $B_{i,N}$ have length 1. Define the $C^\infty(u)$ vector space $\mathcal V^{(N)}$ spanned by $u^{(i)}$ for $1
    \leq i\leq N$, and a linear morphism $T\in\mathrm{End}(\mathcal V^{(N)})$ by 
    \[
    T\left(u^{(r)}\right) = \begin{cases}ru^{(r-1)}, & r\geq 2,\\
        0, & r = 1.\end{cases}
    \]
    Then it is easy to check that
    \[
    \bar{\mathcal D}_i = \frac{T^i}{i!}
    \]
    as linear morphisms on $\mathcal V^{(N)}$. It is easy to see that, by adding suitable coboundaries, we can  require 
    \[
    \qo_{1,N} = b_1(u)u^{(N)}.
    \]
    Now for any $0<i<j<N$, we have 
    \begin{align*}
        \kk{\diff{\qo}{\qt_0}}_{i,j,N} &= \diff{\qo_{j,N}}{u^{(i)}}-\diff{\qo_{i,N}}{u^{(j)}}+\diff{\qo_{i,j}}{u^{(N)}} = 0,\\
        \kk{\diff{\qo}{\qt_1}-\xi_0\qo}_{i,j,N} &= \mathcal D_i\qo_{j,N}-\mathcal{D}_j  \qo_{i,N}+\mathcal D_N \qo_{i,j} = 0,
    \end{align*}
    from which we conclude that 
    \[
    \bar{\mathcal D}_i \qo_{j,N}  = \bar{\mathcal D}_j \qo_{i,N},\quad 0<i<j<N.
    \]
    By taking $i = 1$, we see that 
    \[
    T\qo_{j,N} = \frac{T^j}{j!}\qo_{1,M}.
    \]
    Hence, we must have 
    \[
    \qo_{j,N} = \frac{T^{j-1}}{j!}\qo_{1,M}+b_j(u)u_x,\quad 1<j<N,
    \]
    for some smooth functions $b_j(u)$, since 
     $\ker T = C^\infty(u)u_x$.
     To summarize the discussions so far, we conclude that given a cocycle $\qo\in\hm G^{(N)}$, there uniquely exist functions $a_1(u),b_1(u),\dots, b_N(u)$ such that after adding to $\qo$ suitable coboundaries,
    \[
     \qo = a_1(u)\xi_0\xi_N+b_1(u)u^{(N)}\xi_1\xi_N+\sum_{j=2}^{N-1}\frac{T^{j-1}}{j!}\kk{b_1(u)u^{(N)}}\xi_j\xi_N+\sum_{j=2}^{N-1} b_j(u)u_x \xi_j\xi_N+\hm G^{(N-1)}.
     \] 

     Now for a given smooth function $b(u)$, let us define a differential polynomial $H_N[b]\in\hm G^0\cap \hm G^{(N)}$ such that 
     \[
     \sum_{i\geq 1}u^{(i)}\diff{}{u^{(i)}}\diff{H_N[b]}{u^{(N)}}-\diff{H_N[b]}{u^{(N)}} = -\diff{\Psi_N[b]}{u}.
     \]
     Note that all monomials in $\diff{H_N[b]}{u^{(N)}}$ have at least length 2.
     Denote the cocycle
     \[
     C_N[b] = \diff{\Psi_N[b]}{\qt_0}\xi_N+\kk{\diff{}{\qt_1}-\xi_0}\diff{H_N[b]}{\qt_0},
     \]
     then it follows from a direct computation that 
     \begin{align*}
     C_N[b]_{0,N} &= \diff{\Psi_N[b]}{u}+\sum_{i\geq 1}u^{(i)}\diff{}{u^{(i)}}\diff{H_N[b]}{u^{(N)}}-\diff{H_N[b]}{u^{(N)}} = 0,\\
     C_N[b]_{1,N}&= \diff{\Psi_N[b]}{u_x}+T\diff{H_N[b]}{u^{(N)}} = b(u)u^{(N)}+R,
     \end{align*}
     where all monomials in $R$ have length at least 2. However, $C_N[b]$ is a cocycle in $\hm G^{(N)}$ with $C_N[b]_{0,N} = 0$, hence it follows from \eqref{CL} that  all monomials in $C_N[b]_{1,N}$ must have length 1, and we arrive at $R = 0$. For a same reason, we conclude that 
     \[
     C_N[b]_{i,N} = \frac{T^{i-1}}{i!}\kk{b(u)u^{(N)}},\quad 1\leq i < N.
     \]
     Finally, we notice that 
     \[
     \xi_i\diff{\Psi_N[b]}{\qt_0} = b(u)u_x\xi_i\xi_N+\hm{G}^{(N-1)},
     \]
     therefore it is easy to see that 
     \[
     \qo-\xi_0\diff{\Psi_N[a_1]}{\qt_0}-C_N[b_1]-\sum_{j=2}^{N-1} \xi_i\diff{\Psi_N[b_j]}{\qt_0} \in\hm G^{(N-1)}.
     \]
     The lemma hence follows from induction.
\end{proof}

We are ready to compute $\cok \rho^1_*$ and $\ker\rho^2_*$. 
\begin{Lem}
    The map $\rho^1_*\colon H^1_{d-1}\kk{\hm A[\ql],D_4}\to H^1_{d}\kk{\hm A[\ql],D_1}$
    satisfies $\cok \rho^1_* = 0$ for all $d\geq 2$.
\end{Lem}
\begin{proof}
Let us consider the maps
\[
BH^1_{d-1}\left(\hm A,\diff{}{\qt_0},\diff{}{\qt_1}\right)\xrightarrow{\cong} H^1_{d-1}\kk{\hm A[\ql],D_4}\xrightarrow{\rho^1_*}H^1_d\kk{\hm A[\ql],D_1}\xrightarrow{\cong}BH^1_d\left(\hm A,\diff{}{\qt_0},\diff{}{\qt_1}-\qth^1\right),
\]
and we will still use $\rho^1_*$ to denote the induced map from the leftmost group to the rightmost group.
It follows from Lemma \ref{CM} that if $d=2$ we have $\cok \rho^1_* = 0$. In what follows we consider $d\geq 3$. Using Lemma \ref{CI}, take an arbitrary class 
\[
\qo = \qp_x^{d-2}(f(u)\qth^1)\in BH^1_{d-1}\left(\hm A,\diff{}{\qt_0},\diff{}{\qt_1}\right),
\]
then as cohomology classes in $H^1_d\kk{\hm A[\ql],D_1}$ we have  
\begin{align*}
\rho^1_*[\qo] &= -\left[(u-\ql)\qp_x^{d-1}(f(u)\qth^1)\right]\\
&=-\left[(u-\ql)\qp_x^{d-1}(f(u)\qth^1)\right]+\left[D_1(\qp_x^{d-1}\bar f(u))\right]\\
&=\left[-u\qp_x^{d-1}(f(u)\qth^1)+\qp_x^{d-1}(f(u)u\qth^1)-\qth^1 \qp_x^{d-1}\bar f(u)\right],
\end{align*}
here $\bar f(u)$ is a primitive function of $f(u)$.
Note that we are eliminating the $\ql$-dependence so that we get a cohomology class in 
\[
BH^1_d\left(\hm A,\diff{}{\qt_0},\diff{}{\qt_1}-\qth^1\right).
\]  
Using Lemma \ref{CM}, we know a class in this group is uniquely determined by its $u_x\qth^{d-1}$ coefficient, hence the induced morphism reads
\begin{align*}
    C^\infty(u) \xrightarrow{\cong}BH^1_{d-1}\left(\hm A,\diff{}{\qt_0},\diff{}{\qt_1}\right) &\to BH^1_d\left(\hm A,\diff{}{\qt_0},\diff{}{\qt_1}-\qth^1\right) \xrightarrow{\cong}
 C^\infty(u),\\[1em]
 f(u)&\mapsto (d-1)f(u).
\end{align*}
Therefore for $d\geq 3$, the map $\rho^1_*$ is an isomorphism, and the lemma follows.
\end{proof}
\begin{Lem}
    The map $\rho^2_*\colon H^2_{d}\kk{\hm A[\ql],D_4}\to H^2_{d+1}\kk{\hm A[\ql],D_1}$
    satisfies 
    \[
    \ker\rho^2_*\cong \begin{cases}
			0, & \text{if $d = 2,3,4$ or $d = 2n$ for $n\geq 3$,}\\
            C^\infty(u), & \text{otherwise.}
		 \end{cases}
    \]
\end{Lem}
\begin{proof}
    Using the fact \eqref{CA} and Lemma \ref{CI} as well as Lemma \ref{CO}, we see that any class in 
    \[
BH^2_d\left(\hm A,\diff{}{\qt_0},\diff{}{\qt_1}\right) 
    \]
    can be represented by 
    \[
    \qo = \qth\qp_x^{d-1}(f(u)\qth^1)+\sum_{\substack{i_1+i_2+2=d \\ 0\leq i_1<i_2}}\qp_x^{i_1}(g_{i_1}(u)\qth^1)\qth^{i_2+1}.
    \]
    Consider the differential polynomial 
    \[
    P = -\qth\qp_x^{d-1}(\bar f(u))+\sum_{\substack{i_1+i_2+2=d \\ 0\leq i_1<i_2}}\qp_x^{i_1}(\bar g_{i_1}(u))\qth^{i_2+1},
    \]
    where $\bar f(u)$ and $\bar g_i(u)$ are primitive functions of $f(u)$ and $g_i(u)$. Note that 
    \[
    \diff{P}{\qt_0} = \qo.
    \]
    Then, as cohomology classes in $H^2_{d+1}\kk{\hm A[\ql],D_1}$, we have 
    \begin{align*}
        \rho^2_*[\qo] = \left[u\qp_x\qo-\ql\qp_x\qo\right]
        =\left[u\qp_x\qo-\ql\qp_x\qo+D_1(-P)\right] =  \left[u\qp_x\qo-\left(\diff{}{\qt_1}-\qth^1\right)\qp_x P\right],
    \end{align*}
    and we view the class \[
    A = u\qp_x\qo-\left(\diff{}{\qt_1}-\qth^1\right)\qp_x P
    \]
    as a cohomology class in 
    \[
    BH^2_{d+1}\left(\hm A,\diff{}{\qt_0},\diff{}{\qt_1}-\qth^1\right).
    \]
    Let us assume that $[\qo]\in\ker\rho^2_*$. Using \eqref{CD}, we first separate $A$ into the $\qth$-linear part and $\qth$-independent part. It is easy to see that the $\qth$-linear part reads $\qth B$, where  
    \[
    B = u\qp_x^d(f(u)\qth^1)+\qth^1\qp_x^d(\bar f(u))-\qp_x^{d}(f(u)u\qth^1)
    \]
    defines a class in 
    \[
    BH^1_{d+1}\left(\hm A,\diff{}{\qt_0},\diff{}{\qt_1}-\qth^1\right).
    \]
    Again, the class $[B]$ is determined by its $u_x\qth^d$ coefficient, which is $-d f(u)$. Hence,  we must have $f = 0$ if $\rho^2_*([\qo]) = 0$. In particular, this implies that $\ker\rho^2_* = 0$ if $d = 2$.
    
    Now let us focus on the $\qth$-independent part of $A$, and we continue to use the notations $\xi_i = \qth^{i+1}$. Since $f(u) = 0$, by a straightforward computation, the $\qth$-independent part of $A$ reads 
    \[
    C = u\qp_x\sum\qp_x^{i_1}(g_{i_1}\xi_0)\xi_{i_2}+\xi_0\qp_x\sum\qp_x^{i_1}(\bar g_{i_1})\xi_{i_2}-\qp_x \sum\qp_x^{i_1}(g_{i_1}u\xi_0)\xi_{i_2},
    \]
    where we have omitted the summation range for $i_1$ and $i_2$ for  simplifications. We consider the coboundary
    \begin{align*}
        H =&\, \kk{\diff{}{\qt_1}-\xi_0}\diff{}{\qt_0}\qp_x\sum\qp_x^{i_1}\left(\bar g_{i_1}\right)u^{(i_2)}\\
        =&\,\kk{\diff{}{\qt_1}-\xi_0}\kk{\qp_x\sum\qp_x^{i_1}\left(g_{i_1}\xi_0\right)u^{(i_2)}+\qp_x\sum\qp_x^{i_1}\left(\bar g_{i_1}\right)\xi_{i_2}}\\
        =&\,-\qp_x\sum\qp_x^{i_1}\left(g_{i_1}\xi_0\right)\qp_x^{i_2}(u\xi_0)+\qp_x\sum\qp_x^{i_1}\left(g_{i_1}u\xi_0\right)\xi_{i_2}\\
        &-\xi_0\qp_x\sum\qp_x^{i_1}\left(g_{i_1}\xi_0\right)u^{(i_2)}-\xi_0\qp_x\sum\qp_x^{i_1}\left(\bar g_{i_1}\right)\xi_{i_2},
    \end{align*}
    then we can change the representative $C$ into 
    \[
    \tilde C = C+H = u\qp_x\sum\qp_x^{i_1}(g_{i_1}\xi_0)\xi_{i_2}-\qp_x\sum\qp_x^{i_1}\left(g_{i_1}\xi_0\right)\qp_x^{i_2}(u\xi_0)-\xi_0\qp_x\sum\qp_x^{i_1}\left(g_{i_1}\xi_0\right)u^{(i_2)}.
    \]
    Denote by $z_p:=\qp_x^p(g_p\xi_0)$ and set for $p<q$
    \begin{align*}
    \tilde C_{p,q} &= u\qp_x(z_p\xi_q)-\qp_x(z_p\qp_z^q(u\xi_0))-\xi_0\qp_x\left(z_pu^{(q)}\right)\\
    &= -\qp_xz_p\sum_{r=1}^{q-1}\binom{q}{r}u^{(q-r)}\xi_r-z_p\sum_{r=1}^q\binom{q+1}{r}u^{(q+1-r)}\xi_r.
    \end{align*}
    According to the definition \eqref{CH}, we have 
    \[
    \Psi_N[a] = \frac{a(u)}{2(N+1)}\sum_{r=1}^N\binom{N+1}{r}u^{(r)}u^{(N+1-r)} + R,
    \]
    where all monomials in $R$ have length at least 3. Therefore, 
    \[
    \diff{\Psi_N[a]}{\qt_0} = \frac{a(u)}{N+1}\sum_{r=1}^N\binom{N+1}{r}u^{(N+1-r)}\xi_r + Q,
    \]
    where all monomials in $Q$ have length at least 2.
    Now we note that $\tilde C_{p,q}$ has no length 0 monomials and consider the length 1 monomials which are given by 
    \[
    -g_p\xi_{p+1}\sum_{r=1}^{q-1}\binom{q}{r}u^{(q-r)}\xi_r-g_p\xi_p\sum_{r=1}^q\binom{q+1}{r}u^{(q+1-r)}\xi_r,
    \]
    from which we conclude that the cocycle 
    \[
\tilde C_{p,q}-(q+1)\diff{\Psi_q[g_p]}{\qt_0}\xi_p-q\diff{\Psi_{q-1}[g_p]}{\qt_0}\xi_{p+1}
    \]
    has neither length 0 monomilas nor length 1 monomials. We then apply Lemma \ref{CK} to conclude that 
    \[
    [\tilde C_{p,q}] = \left[(q+1)\diff{\Psi_q[g_p]}{\qt_0}\xi_p+q\diff{\Psi_{q-1}[g_p]}{\qt_0}\xi_{p+1}\right]
    \]
    as cohomology classes in 
    \[BH^2_{d+1}\left(\hm G,\diff{}{\qt_0},\diff{}{\qt_1}-\xi_0\right).\]
    Now define the class 
    \[
    E_r[g] = (d-r)\diff{\Psi_{d-1-r}[g]}{\qt_0}\xi_r, 
    \] 
    then we have 
    \[
    [\tilde C_{p,d-1-p}] = \left[E_p[g_p]\right]+\left[E_{p+1}[g_p]\right],
    \]
    and it follows from the proof of Lemma \ref{CP} that 
    \[
    [\tilde C] = [E_0[g_0]]+\sum_{r\geq 2}[E_r[g_{r-1}+g_r]]
    \]
    as cohomology classes.

    Now if $d = 3$, we must have $g_0 = 0$ and hence $\ker \rho^2_* = 0$. If $d = 2n+3$ for some $n\geq 1$, we see that 
    \[[\tilde C] = [E_0[g_0]]+\sum_{r= 2}^{n}[E_r[g_{r-1}+g_r]],\]
    from which it follows that 
    \[
    g_0 = 0,\quad g_{r-1}+g_r = 0,\quad 2\leq r\leq n.
    \]
    This implies that $g_r = (-1)^{r-1}g_1(u)$, where $g_1(u)$ is an arbitrary smooth function.

    If $d = 2n+2$ for some $n\geq 1$, then it follows that 
    \[[\tilde C] = [E_0[g_0]]+\sum_{r= 2}^{n-1}[E_r[g_{r-1}+g_r]]+E_n[g_{n-1}],\]
    and we arrive at 
    \[
    g_r = 0,\quad 0\leq r\leq n-1.
    \]
    The lemma is proved.
\end{proof}
After summarizing all the constructions and computations above, it follows that Theorem \ref{CN} holds true.

\section{Conclusion}
\label{AL}

In this paper, we study the deformation of scalar semisimple generalized bi-Hamiltonian structures. The moduli space of such structures is parametrized by two smooth univariate
  functions $(f(u),g(u))$, and we see that only when these two functions satisfy certain relations, do the corresponding generalized bi-Hamiltonian structure possibly admit non-trivial deformations.

A particularly interesting case is the scalar generalized bi-Hamiltonian structure of viscous type, which admits a non-trivial deformation starting from differential degree 2, which is a new phenomenon that cannot be studied within the framework of ordinary bi-Hamiltonian structures. This provides a possible Dubrovin-Zhang-theoretical viewpoint towards the study of cohomological field theories on the space of Riemann surfaces with boundaries. Furthermore, it follows from the results of this paper that the generalized bi-Hamiltonian structures of viscous type may admit a nice higher-dimensional analog, and one can expect that the moduli space of deformations of an $n$-component generalized bi-Hamiltonian structure of viscous type (once defined) is parametrized by $n$ univariate functions, which play a same role as central invariants in the theory of ordinary bi-Hamiltonian structures. We will study this problem elsewhere.

Another interesting possibility is to extend the construction of super tau-covers \cite{liu2020super,liu2022variational} to the bi-flat F-manifolds. Indeed,  associated with any bi-flat F-manifold, there naturally exists a generalized bi-Hamiltonian structure which commute with the Principal Hierarchy of that manifold \cite{lorenzoni2026generalised}. And the construction of super-tau covers will require the study of Virasoro symmetries (or their certain generalizations) for a bi-flat F-manifold, and investigate how it may control the higher genus theory of semisimple bi-flat F-manifolds. We will study this in a separate paper.


\end{document}